\documentclass[a4paper]{article}
\usepackage{amsmath}
\usepackage{amsthm}
\usepackage{amssymb}
\usepackage{graphicx}
\usepackage{authblk}
\usepackage{hyperref}
\usepackage{multirow}
\usepackage{cite}
\usepackage{comment}
\usepackage{overpic}

\newtheorem{theorem}{Theorem}[section]
\newtheorem{proposition}[theorem]{Proposition}
\newtheorem{lemma}[theorem]{Lemma}

\newtheorem{remark}{Remark}

\newenvironment{vf}{\left\{\begin{array}{rcl}}{\end{array}\right.}

\newcommand{\includegraph}[2][]{\ifnum\pdfoutput=0\includegraphics[#1]{#2.eps}\else\includegraphics[#1]{#2.pdf}\fi}
\newcommand{\N}{\mathbb{N}}

\newcommand{\R}{\mathbb{R}}

\author[1]{Renato Huzak}
\author[2]{Ansfried Janssens\footnote{Corresponding author, {\tt ansfried.janssens@uhasselt.be}}}
\author[3]{Otavio Henrique Perez}
\author[4]{Goran Radunovi\'{c}}

\affil[1,2,3]{Hasselt University, Campus Diepenbeek, Agoralaan Gebouw D, 3590 Diepenbeek, Belgium}
\affil[3]{Universidade de S\~{a}o Paulo (USP), Instituto de Ci\^{e}ncias Matem\'aticas e de Computa\c{c}\~{a}o (ICMC). Avenida Trabalhador S\~{a}o Carlense, 400, CEP 13566-590, S\~{a}o Carlos, S\~{a}o Paulo, Brazil.}
\affil[4]{University of Zagreb, Faculty of Science, Horvatovac 102a, 10000 Zagreb, Croatia}

\affil[1]{{\tt renato.huzak@uhasselt.be}}
\affil[2]{{\tt ansfried.janssens@uhasselt.be}}
\affil[3]{{\tt otavio.perez@icmc.usp.br}}
\affil[4]{{\tt goran.radunovic@math.hr}}

\title{Fractal analysis of canard cycles and slow-fast Hopf points in piecewise smooth Li\'{e}nard equations}

\date{}

\begin{document}
\maketitle

\begin{abstract}
The main goal of this paper is to give a complete fractal analysis of piecewise smooth (PWS) slow-fast Li\'{e}nard equations. For the analysis, we use the notion of Minkowski dimension of one-dimensional orbits generated by slow relation functions. More precisely, we find all possible values for the Minkowski dimension near PWS slow-fast Hopf points and near bounded balanced crossing canard cycles. We study fractal properties of the unbounded canard cycles using PWS classical Li\'{e}nard equations. We also show how the trivial Minkowski dimension implies the non-existence of limit cycles of crossing type close to Hopf points. This is not true for crossing limit cycles produced by bounded balanced canard cycles, i.e. we find a system undergoing a saddle-node bifurcation of crossing limit cycles and a system without limit cycles (in both cases, the Minkowski dimension is trivial). We also connect the Minkowski dimension with upper bounds for the number of limit cycles produced by bounded canard cycles.
\end{abstract}

\tableofcontents

\section{Introduction}\label{Section-Introduction}

The main purpose of this paper is to give a fractal classification of piecewise smooth (PWS) continuous slow-fast Li\'enard equations
\begin{align}\label{PWSLienard-Intro}
  \begin{cases}
  \dot x=y-F(x) ,\\ 
  \dot y=\epsilon G(x), \end{cases}
  \end{align}
  with
  \begin{align}\label{PWSLienard-Intro-FG}
  F(x)=\begin{cases}
  F_-(x), \ x\le 0,\\ 
  F_+(x), \ x\ge 0, \end{cases}
  \quad
  G(x)=\begin{cases}
  G_-(x), \ x\le 0,\\ 
  G_+(x), \ x\ge 0, \end{cases}
  \end{align}
where $\epsilon\ge 0$ is a singular perturbation parameter kept small, $F_\pm$ and $G_\pm$ are $C^\infty$-smooth functions, $F_\pm(0)=F_\pm'(0)=0$ and $G_\pm(0)=0$. The set $\Sigma = \{x=0\}$ is called the switching
line or switching manifold. For the classification, we will use the notion of Minkowski dimension (always equal to the box dimension \cite{Falconer,tricot}) of one-dimensional monotone orbits generated by so-called slow relation (or entry-exit) function (see Section \ref{sec-Motivation}). We refer to \cite{Benoit,Dbalanced,BoxNovo} and references therein for the definition of the notion of slow relation function in smooth planar slow-fast systems.
\smallskip

One of the important properties of such one-dimensional monotone orbits is their density. The density is usually measured by calculating the length of
$\delta$-neighborhood of orbits as $\delta\to 0$ and comparing the length with $\delta^{1-s}$, $0\le s\le 1$. In this way we obtain the Minkowski dimension of orbits, taking values between $0$ and $1$ (for more details, see Section \ref{Minkowski-def-separated}). The bigger the Minkowski dimension of the orbits, the higher the density of orbits. Following \cite{EZZ,ZupZub,MRZ}, the Minkowski dimension of orbits generated by the Poincar\'{e} map near foci, limit cycles, homoclinic loops, etc., is closely related to the number of limit cycles produced in bifurcations (roughly speaking, the bigger the Minkowski dimension,
the more limit cycles can be born). 

Similarly, in smooth planar slow-fast systems, the Minkowski dimension of orbits generated by a slow relation function plays an important role in detecting the codimension of singular Hopf bifurcations in a coordinate-free way \cite{BoxNovo}, finding the maximum number of limit cycles produced by
canard cycles \cite{BoxRenato,BoxDomagoj,BoxVlatko}, etc. For a more detailed motivation we refer the interested reader to \cite[Section 1]{BoxNovo} and \cite[Section 1]{MinkLienard}. Since these papers deal only with smooth slow-fast systems, it is natural to ask whether we can use similar methods to study fractal properties of \textit{PWS slow-fast systems}.

\begin{figure}[htb]
	\begin{center}
		\includegraphics[width=2.3cm,height=3.5cm]{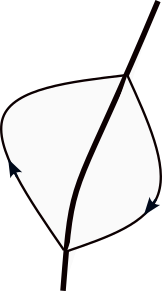}
		{\footnotesize
        \put(-13,98){$\Sigma$}
        \put(-80,78){$X_-$}
        \put(33,78){$X_+$}
             }
         \end{center}
 \caption{A crossing periodic orbit where $\Sigma$ is the switching manifold.}
	\label{fig-crossing-example}
\end{figure}

Piecewise smooth systems \cite{filippov1988differential} are an active field of recent research. The determination of crossing limit cycles, for example, is an important problem in PWS theory in the plane (see \cite{carmona2023a,llibre2013a,LlibreOrd,HuanYang2,freire2013a} and references therein). Such cycles intersect the switching manifold $\Sigma$ at points where the vector fields $X_-$ and $X_+$ point in the same direction relative to $\Sigma$ (see Figure \ref{fig-crossing-example}).

In this paper we are interested in the following limit periodic sets of the PWS slow-fast system \eqref{PWSLienard-Intro} (we also refer to Figure \ref{fig-Motivation} in Section \ref{sec-Motivation}). Assume that the curve of singularities of \eqref{PWSLienard-Intro} contains a normally attracting branch $\{y=F_+(x),x>0\}$ and a normally repelling branch $\{y=F_-(x),x<0\}$. Then the balanced canard cycle $\Gamma_{\hat y}$, with $\hat y>0$ when $\epsilon = 0$, consists of a portion of the normally attracting branch, a portion of the normally repelling branch and a horizontal fast orbit at level $y=\hat y$. These canard cycles may produce crossing limit cycles after a perturbation of \eqref{PWSLienard-Intro} if the slow dynamics of \eqref{PWSLienard-Intro} defined along the curve of singularities points from the attracting branch to the repelling branch and it has a regular extension through the origin $(x,y)=(0,0)$ (see Section \ref{section-applications}). This can be done by merging two smooth slow-fast Hopf points \cite{DRbirth} into a so-called PWS slow-fast Hopf point located at the origin $(x,y)=(0,0)$ (see  assumption \eqref{assum1} in Section \ref{sec-Motivation}). The PWS slow-fast Hopf point is the limit of $\Gamma_{\hat y}$ when $\hat y\to 0$. If $\hat y\to \infty$, we can have an unbounded canard cycle, which will be denoted by $\Gamma_\infty$, consisting of the curve of singularities and a part at infinity (see Section \ref{sectionunbounded}).

The main goal is to give a \textit{complete} fractal classification (i.e., we find all possible Minkowski dimensions) of the PWS slow-fast Hopf point and $\Gamma_{\hat y}$ described in the previous paragraph, related to the PWS continuous Li\'enard family \eqref{PWSLienard-Intro}. The fractal classification of the PWS slow-fast Hopf point is given in Theorems \ref{thm1} and \ref{thm2} in Section \ref{sectionHopf}, whereas the possible Minkowski dimensions of $\Gamma_{\hat y}$ are given in Theorem \ref{thm3} in Section \ref{sectionbounded}. We assume that $\Gamma_{\hat y}$ is balanced, that is, the slow divergence integral computed along the portion of the curve of singularities contained in $\Gamma_{\hat y}$ is zero (see Section \ref{sectionbounded} and \cite{Dbalanced}). In Theorem \ref{thm4} in Section \ref{sectionunbounded} we compute Minkowski dimensions near $\Gamma_\infty$ when \eqref{PWSLienard-Intro} is a PWS classical Li\'enard system (that is, $F_\pm$ are polynomials of degree $n+1$, $n\ge 1$, and $G_\pm$ are linear). Theorems \ref{thm1} to \ref{thm4} are proven in Section \ref{section-proof-all}.

In Section \ref{section-applications}, the link between the Minkowski dimensions computed in Section \ref{section-proof-all} and the number of crossing limit cycles of a perturbation of \eqref{PWSLienard-Intro} near $\Gamma_{0}=\{(0,0)\}$, $\Gamma_{\hat{y}}$ and $\Gamma_{\infty}$ is addressed (see system \eqref{eq-pws-lienard-general} in Section \ref{section-applications}). We focus our study to the case in which the Minkowski dimensions are trivial (that is, equal zero). Geometrically speaking, close to a PWS Hopf point, trivial Minkowski dimension of orbits tending to $\Gamma_{0}$ means that the connection between center manifolds are broken in the blow-up locus, so we cannot expect crossing limit cycles (see Section \ref{sec-blow-up} and Figure \ref{fig-pws-lienard}). However, trivial Minkowski dimension of orbits tending to $\Gamma_{\hat{y}}$ does not imply the absence of limit cycles. Indeed, we present an example in which the system undergoes a saddle-node bifurcation of limit cycles (see Section \ref{section-bounded-LC}). Finally, in Section \ref{sec-LC-unbound} we give examples in which the Minkowski dimension of orbits tending to $\Gamma_{\infty}$ is trivial, but, nevertheless, one can expect crossing limit cycles. The number of limit cycles related to higher Minkowski dimensions is a topic for future study (see Remark \ref{remark-nonzeroMD} in Section \ref{section-applications} for some results in that direction).

In the smooth setting, that is, when the functions $F$ and $G$ in \eqref{PWSLienard-Intro-FG} are $C^\infty$-smooth, we deal with a smooth slow-fast Hopf point at the origin $(x,y)=(0,0)$ and the following discrete set of values of the Minkowski dimension can be produced (see \cite{BoxNovo}): $\frac{1}{3},\frac{3}{5},\frac{5}{7},\dots,1$. From these values, which can be computed numerically \cite{BoxNovo}, we can read upper bounds for the number of limit cycles produced by the smooth slow-fast Hopf point (for more details see \cite[Theorem 3.4]{BoxNovo}). Theorems \ref{thm1} and \ref{thm2} imply that the PWS slow-fast Hopf point produces infinitely many new values: $0,\frac{1}{2},\frac{2}{3},\frac{3}{4},\dots$. We strongly believe that they give information about the number of limit cycles produced by the PWS slow-fast Hopf point. This is a topic of further study.

Similarly, besides old values of the Minkowski dimension when $F$ is a polynomial of even degree $n+1$ and $G$ is linear ($\frac{1}{2},\frac{3}{4},\dots,\frac{n-2}{n-1}$, see \cite[Remark 1]{MinkLienard}), in the piecewise smooth setting, $\Gamma_\infty$ produces the following new values: $0,\frac{4}{5},\frac{6}{7},\dots,\frac{n-1}{n}$. We refer to Theorem \ref{thm4} for more details. 

The main reason why we assume that $F_\pm$ and $G_\pm$ are $C^\infty$-smooth is because we want to detect all possible Minkowski dimensions of orbits. We need higher-order Taylor expansions (i.e., higher degrees of smoothness of $F_\pm$ and $G_\pm$) in order to find larger Minkowski dimensions of orbits (see Sections \ref{proofThm1} and \ref{proofThm3}).

Let us highlight the differences between our approach and those already presented in the literature concerning PWS slow-fast systems. In \cite{2016Roberts, 2024CX} the authors studied the existence of crossing canard limit cycles in piecewise smooth Li\'enard equations, in such a way that the origin is a corner point of the critical manifold (or curve of singularities) positioned in $\Sigma$, in the sense that $F'(0)$ in \eqref{PWSLienard-Intro} is not well defined. Moreover, the critical manifold of the models studied in such references presents a ``van der Pol - like'' shape. On the other hand, in \cite{2013DFHPT} the authors also address the existence of canard limit cycles, but considering a three-zoned piecewise smooth Li\'enard equation instead. In \cite{2023CFGT, 2016FGDKT} the authors studied the existence of canard cycles in four-zoned and three-zoned piecewise linear (PWL) systems, respectively.

In all those references, \cite{2023CFGT, 2024CX, 2013DFHPT, 2016FGDKT, 2016Roberts}, the authors fixed a linear function $G$ in \eqref{PWSLienard-Intro} (that is, $G_{-} = G_{+}$) and defined $F$ in a piecewise smooth way. Moreover, in their models, the critical manifold loses smoothness in the intersection with the switching manifold. On the other hand, in this paper, we allow both $F$ and $G$ to be defined in a piecewise smooth way. Moreover, in the study of the Hopf point and the bounded canard cycle, we do not require $G$ to be linear. In addition, the intersection between the critical manifold and $\Sigma$ is not a corner point. Finally, our main tools are Fractal Geometry, Slow Divergence Integrals and Slow-Relation Functions, which were not used in those previous references.

A connection between sliding canard cycles in regularized PWS systems and the slow divergence integral can be found in \cite{RHKK2023}.

\section{Minkowski dimension}\label{Minkowski-def-separated}

Let $U\subset\mathbb R^N$ be a bounded set. One defines its $\delta$-neighborhood  (or $\delta$-parallel body) as
$
U_\delta:=\{x\in\mathbb R^N \ | \ \text{dist}(x,U)\le\delta\}
$, where $\text{dist}(x,U)$ denotes the euclidean distance from $x$ to the set $U$. Denote the Lebesgue measure of $U_\delta$ by $|U_\delta|$. For $s\ge0$, we introduce the lower $s$-dimensional Minkowski content of $U$
$$
\mathcal M_*^s(U)=\liminf_{\delta\to 0}\frac{|U_\delta|}{\delta^{N-s}},
$$
and similarly, the upper $s$-dimensional Minkowski content $\mathcal M^{*s}(U)$ (replacing $\liminf_{\delta\to0}$ with $\limsup_{\delta\to0}$ above). We then define the lower and upper Minkowski (or box-counting, since they always coincide) dimensions of $U$ as:
$$
\underline\dim_BU=\inf\{s\ge0 \ | \ \mathcal M_*^s(U)=0\}, \ \overline\dim_BU=\inf\{s\ge0 \ | \ \mathcal M^{*s}(U)=0\}.
$$

When the upper and lower dimensions coincide, we refer to their common value as the Minkowski dimension of $U$, denoted by $\dim_BU$. For a comprehensive treatment of Minkowski dimension, we direct the reader to \cite{Falconer,tricot} and the references therein. Furthermore, if there exists a $d$ such that $0<\mathcal M_*^d(U)\le\mathcal M^{*d}(U)<\infty$, we say that $U$ is Minkowski nondegenerate in which case, $d=\dim_B U$ necessarily.

Consider a bi-Lipschitz mapping $\Phi:U \subset \mathbb{R}^{N}\rightarrow \mathbb{R}^{N_1}$, i.e., there exists a constant $\rho>0$ such that
$$
\rho\left\|x-y\right\|\leq\left\|\Phi(x)-\Phi(y)\right\| \leq\frac{1}{\rho} \left\|x-y\right\|
$$
for all $x,y\in U$. Then it is well-known that
$$
\underline\dim_{B}U=\underline\dim_{B}\Phi(U), \ \overline\dim_{B}U=\overline\dim_{B}\Phi(U).
$$

Moreover, if $U$ is Minkowski nondegenerate, then $\Phi(U)$ is also Minkowski nondegenerate (refer to \cite[Theorem 4.1]{ZuZuR^3}).

\smallskip

We also introduce here some notation used throughout this paper. For two sequences of positive real numbers $(a_l)_{l\in\mathbb{N}}$ and $(b_l)_{l\in\mathbb{N}}$ converging to zero, we write $a_l\simeq b_l$ as $l\to\infty$
if there exists a small positive constant $\rho$ such that $\frac{a_l}{b_l}\in[\rho,\frac{1}{\rho}]$ for all $l\in\mathbb{N}$.

Note that the Minkowski dimension has proven to be a useful tool in fractal analysis of various dynamical systems which enables one to extract information about the cyclicity of the system directly from analyzing the Minkowski dimension of one of its orbits \cite{zbMATH02196683,GoranInf} or even by looking just at the Minkowski dimension of a discrete orbit generated by the suitable Poincar\'e map \cite{EZZ,ZupZub} or even Dulac map \cite{zbMATH07307367}.
Furthermore, since the Minkowski dimension is always equal to the box dimension \cite{Falconer} which can be effectively computed numerically \cite{WU2020100106,FREIBERG2021105615,MEISEL19971565,10.1007/978-3-030-64616-5_8,RuizDeMiras2020,Panigrahy2019DifferentialBC,10.1063/5.0160394,BoxNovo}, it is natural to expect that numerical methods for determining cyclicity via the Minkowski dimension can be developed which puts further value to the results in our paper. 

It was also shown that the Minkowski dimension is useful in providing a novel tool for formal and analytic classification of parabolic diffeomorphism in the complex plane.
Even for the formal classification of parabolic germs one first needs to extend the definition of the Minkowski dimension either as in \cite{zbMATH06224533} or alternatively also look at higher order terms in the asymptotic series of the $\delta$-neighborhood of the orbit as $\delta$ tends to zero \cite{zbMATH07584629}. The latter approach is closely connected to the theory of complex (fractal dimensions) and associated fractal zeta functions introduced by Lapidus and van Frankenhuijsen \cite{LF12} for subsets of $\R$ and then extended to the general case of subsets of $\R^N$ in \cite{Goran}.
Furthermore, in order to tackle the analytic classifications of parabolic germs one needs to further extend and adapt the theory of complex dimensions as in \cite{KMRR25}.

Finally, note that, in contrast to the Minkowski dimension, the Hausdorff dimension would not extract us any relevant information from orbits of dynamical systems. The reason for this stems from the countable stability of the Hausdorff dimension which renders all of the orbits of the dynamical systems to have either dimension 1 in the continuous case, or 0 in the discrete case. On the other hand, the lack of the countable stability of the Minkowski dimension is exactly the reason which makes it interesting and useful for the fractal analysis of orbits of dynamical systems. Of course, as it is well known, an attractor of a dynamical system can have nontrivial Hausdorff dimension (strange attractors such as Lorenz or H\'enon, etc). However, in all cases mentioned above, as well as in this paper the attractor is either a point (possibly at infinity) or a piecewise smooth curve, hence of trivial Hausdorff dimension.

\section{PWS slow-fast Li\'enard systems and statement of results}
\label{sec-Motivation}
We consider a PWS slow-fast Li\'enard equation
\begin{align}\label{PWSLienard}
 X_-: \begin{cases}
  \dot x=y-F_-(x) ,\\ 
  \dot y=\epsilon G_-(x), \end{cases}
  \   \text{for }x\le 0,\quad
 X_+: \begin{cases}
  \dot x=y-F_+(x) ,\\ 
  \dot y=\epsilon G_+(x), \end{cases}  \  \text{for }x\ge 0,
  \end{align}
where $0<\epsilon\ll 1$ is a singular perturbation parameter and $F_\pm$ and $G_\pm$ are $C^\infty$-smooth functions. We assume that $X_-$ and $X_+$ have a slow-fast Hopf point at the origin $(x, y) = (0,0)$, that is, they satisfy (see also \cite[Definition 1.1]{DRbirth})
\begin{equation}
\label{assum1}
F_\pm(0)=F_\pm'(0)=G_\pm(0)=0, \ F_\pm''(0)>0 \text{ and } G_\pm'(0)<0. 
\end{equation} 

We say that system \eqref{PWSLienard} satisfying \eqref{assum1} has a PWS slow-fast Hopf point at $(x,y)=(0,0)$. For $\epsilon=0$, system \eqref{PWSLienard} has the curve of singularities
$$ S=\{(x,F_-(x)) \ | \ x< 0\}\cup \{(0,0)\}\cup\{(x,F_+(x)) \ | \ x> 0\}.$$

We denote by $ S_-$ (resp. $ S_+$) the branch of $S$ contained in $x< 0$ (resp. $x> 0$). We refer to Figure \ref{fig-Motivation}. From \eqref{assum1}, it follows that (near the PWS slow-fast Hopf point) $ S_-$ (resp. $ S_+$) consists of normally repelling (resp. attracting) singularities. Then we can define the slow vector field of \eqref{PWSLienard} along $S$, near $(x,y)=(0,0)$, as the following PWS vector field
\begin{equation}\label{PWSslowdyn}
X_-^s: \ \frac{dx}{d\tau}=\frac{G_-(x)}{F_-'(x)}, \ x\le 0, \ \ \quad X_+^s: \ \frac{dx}{d\tau}=\frac{G_+(x)}{F_+'(x)}, \ x\ge 0,
\end{equation}
 where $\tau=\epsilon t$ is the slow time ($t$ denotes the fast time in \eqref{PWSLienard}). Its flow is called the slow dynamics. Using \eqref{assum1}, it is clear that \eqref{PWSslowdyn} has a removable singularity in $x=0$ and the slow dynamics is regular and it points from the attracting branch $S_+$ to the repelling branch $ S_-$. Notice that the slow vector field \cite[Chapter 3]{DDR-book-SF} of the smooth slow-fast system $X_-$ (resp. $X_+$) along $ S_-$ (resp. $ S_+$) is given by $X_-^s$ (resp. $X_+^s$) defined in \eqref{PWSslowdyn}. See also \cite{DHGener}.
\smallskip

In this paper we focus on fractal analysis of $3$ different types of limit periodic sets of \eqref{PWSLienard}, for $\epsilon=0$ (see Figure \ref{fig-Motivation}): (a) the PWS slow-fast Hopf point at $(x,y)=(0,0)$ (Section \ref{sectionHopf}), (b) bounded canard cycles 
$\Gamma_{\hat y}$, $\hat y>0$, consisting of the fast horizontal orbit of \eqref{PWSLienard} passing through the point $(0,\hat y)$ and the portion of $S$ between the $\omega$-limit point $(\omega(\hat y),\hat y)\in S_+$ and the $\alpha$-limit point $(\alpha(\hat y),\hat y)\in S_-$ of that orbit (Section \ref{sectionbounded}), and (c) an unbounded canard cycle consisting of $S$ and a part at infinity (Section \ref{sectionunbounded}). When we deal with the canard cycles in (b) and (c), we need some additional assumptions on the functions $F_\pm$ and $G_\pm$:
 \begin{equation}
\label{assum2}
F_-'(x)<0, \ G_-(x)>0, \ \forall x\in L_-, \quad F_+'(x)>0, \ G_+(x)<0, \ \forall x\in L_+, 
\end{equation}
where $L_-=[\alpha(\hat y),0[$ and $L_+=]0,\omega(\hat y)]$ in case (b) and $L_-=]-\infty,0[$ and $L_+=]0,\infty[$ in case (c). The assumptions in \eqref{assum2} imply that the slow vector field \eqref{PWSslowdyn} is well-defined on the closure $\overline {L_-\cup L_+}$ and it has no singularities.

\begin{figure}[htb]
	\begin{center}
		\includegraphics[width=6.9cm,height=6.5cm]{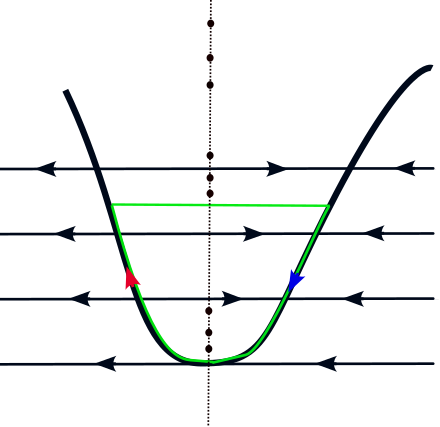}
		{\footnotesize
        \put(-29,148){$S_+$}
        \put(-111,188){$x=0$}
        \put(-71,99){$\Gamma_{\hat y}$}
        \put(-100,147){$y_0$}
        \put(-100,159){$y_1$}
        \put(-100,173){$y_2$}
\put(-100,119){$y_0$}
        \put(-100,106){$y_1$}
        \put(-100,99){$y_2$}
        \put(-101,49){$y_0$}
        \put(-101,40){$y_1$}
        \put(-101,33){$y_2$}
        \put(-59,62){$I_+$}
        \put(-152,62){$I_-$}
        \put(-44,94){$(\omega (\hat y),\hat y)$}
        \put(-191,94){$(\alpha (\hat y),\hat y)$}
         \put(-158,133){$S_-$}
             }
         \end{center}
 \caption{The phase portrait of \eqref{PWSLienard} for $\epsilon=0$, with indication of the slow dynamics along the curve of singularities $S$. Orbits $U=\{y_0,y_1,\dots\}$ generated by the slow relation function $H$ can converge to the PWS slow-fast Hopf point $(x,y)=(0,0)$, a canard cycle $\Gamma_{\hat y}$ (green) or the unbounded canard cycle.}
	\label{fig-Motivation}
\end{figure}
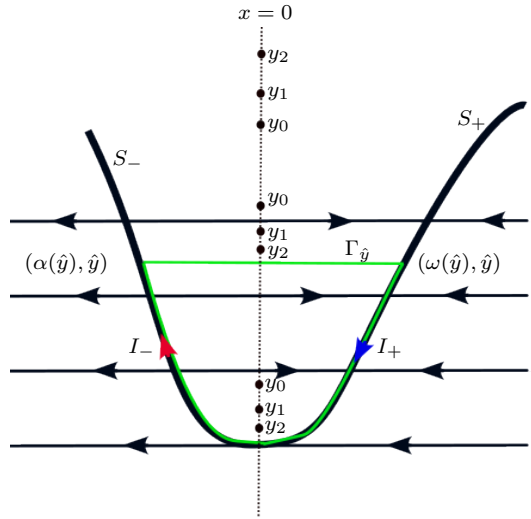

The canard cycles considered throughout this paper are, in fact, crossing canard cycles according to Filippov's convention \cite{filippov1988differential}. More precisely, when one deals with piecewise smooth vector fields, one can define sewing and sliding regions in the switching locus $\Sigma = \{x = 0\}$, which are given by
\begin{equation*}
\begin{array}{ccc}
  \Sigma^{w} & = & \big{\{}(x,y)\in\Sigma \ ; \ (y - F_{-}(x))(y - F_{+}(x)) > 0\big{\}},  \\
\Sigma^{s} & = & \big{\{}(x,y)\in\Sigma\ ; \  (y - F_{-}(x))(y - F_{+}(x)) < 0\big{\}}, 
\end{array}
\end{equation*}
respectively. It follows directly from assumptions in \eqref{assum1} that $\Sigma^{s} = \emptyset$ and $\Sigma^{w} = \Sigma\backslash\{0\}$.

Slow divergence integrals (see \cite[Chapter 5]{DDR-book-SF} and \cite{DHGener}) play an important role in fractal analysis of the limit periodic sets defined above. The slow divergence integral of $X_-$ (resp. $X_+$) associated with the segment $[\alpha( y),0]$ (resp. $[0,\omega(y)]$) are given by
\begin{equation}\label{PWSSDI}
I_-(y):=-\int_{\alpha(y)}^{0}\frac{F'_-(x)^2}{G_-(x)}dx<0, \ I_+(y):=-\int_{\omega(y)}^{0}\frac{F_+'(x)^2}{G_+(x)}dx<0,
\end{equation}
respectively, where $\alpha(y)<0$, $F_-(\alpha(y))=y$, $\omega(y)>0$ and $F_+(\omega(y))=y$. The argument $y>0$ of $I_\pm$ is close to $y=0$ (case (a)), $y=\hat y$ (case (b)) or large enough (case (c)). It is not difficult to see that $I_\pm'(y)<0$ and $I_\pm(y)\to 0$ as $y$ tends to $0$. We also define
\begin{equation}
    \label{SDI-total}
    I(y):=I_+(y)-I_-(y).
\end{equation}

Our goal is to compute the Minkowski dimension of orbits generated by so-called slow relation function $H$ (or its inverse) defined by
\begin{equation}\label{slow-relation-def}
I_-(H(y))=I_+(y).
\end{equation}

See \cite[Section 4]{Dbalanced} for more details concerning the slow relation function in the framework of smooth slow-fast systems. We denote by $U$ the orbit of $y_0>0$ by $H$, that is, $U=\{y_l=H^l(y_0) \ | \ l\in\mathbb N\}$, in which $H^l$ is the $l$-fold composition of $H$. In Section \ref{sectionHopf} (resp. Sections \ref{sectionbounded} and \ref{sectionunbounded}) we consider orbits $U$ that tend to $0$ (resp. $\hat y$ and $\infty$). 
\smallskip

\subsection{Fractal analysis of the PWS slow-fast Hopf point}\label{sectionHopf}

In this section we consider \eqref{PWSLienard} in a small neighborhood of the PWS slow-fast Hopf point $(x,y)=(0,0)$. Since the integrals $I_-$ and $I_+$ are (continuous) decreasing functions and tend to zero as $y\to 0$, it is clear that, for each $y>0$ small enough, there is a unique $H(y)>0$ such that \eqref{slow-relation-def} holds. Analogously, for each $y>0$ small enough, there is a unique $H^{-1}(y)>0$ such that $I_-(y)=I_+(H^{-1}(y))$. Furthermore, we also assume that there is a small $y^*>0$ such that $I$ defined in \eqref{SDI-total} is nonzero in the open interval $]0,y^*[$.

Now, given $y_0\in ]0,y^*[$, if $I>0$ (resp. $I<0$) on $]0,y^*[$ then we denote by $U$ the orbit $\{y_0,y_1,y_2,\dots\}$ defined by
\begin{center}
$I_-(y_{l+1})=I_+(y_l)$ \ (resp. $I_-(y_{l})=I_+(y_{l+1})$), \ with $l\ge 0$.
\end{center}

Observe that $U$ is the orbit of $y_0$ by $H$ (resp. $H^{-1}$) and it tends monotonically to $0$ as $l\to\infty$. 
Conversely, if $y_0>0$ and the orbit of $y_0$ by $H$ (resp. $H^{-1}$) tends monotonically to $0$, then $I>0$ (resp. $I<0$) in the open interval $]0,y^*[$, for a small $y^*>0$.

\begin{theorem}\label{thm1}
Consider a PWS slow-fast Li\'enard system \eqref{PWSLienard} and assume that \eqref{assum1} is satisfied. Given $y_0\in ]0,y^*[$, let $U$ be the orbit with the initial point $y_0$ (as defined above). Then $\dim_B U$ exists and 
\begin{equation}
    \label{formula-BOX}
   \dim_BU\in\left \{\frac{m-1}{m+1} \ | \ m=1,2,\dots\right\}\cup\{1\}.
\end{equation}

If $\dim_BU\ne 0,1$, then $U$ is Minkowski nondegenerate. These results do not depend on the choice of the initial point $y_0\in ]0,y^*[$.  
\end{theorem}

\begin{remark}\label{remark-old-new}
  From \eqref{formula-BOX} in Theorem \ref{thm1} it follows that $\dim_BU$ can take the following 
discrete set of values: $0,\frac{1}{3},\frac{1}{2},\frac{3}{5},\frac{2}{3},\frac{5}{7},\dots,1$. We point out that the values $\frac{1}{3},\frac{3}{5},\frac{5}{7},\dots$ ($m$ even) and $1$ are found near smooth slow-fast Hopf points (see \cite{BoxNovo}). Besides these old values of the Minkowski dimension, the PWS slow-fast Hopf point in \eqref{PWSLienard} also produces infinitely many new values: $0,\frac{1}{2},\frac{2}{3},\dots$ ($m$ odd). See also Theorem \ref{thm2} and Remark \ref{remark-normal}.
\end{remark}
The proof of Theorem \ref{thm1} goes as follows (for more details we refer to Section \ref{proofThm1}). Firstly, using assumptions \eqref{assum1} we can write $$F_\pm(x)=x^2f_\pm(x),$$
where $f_-$, $f_+$ are $C^\infty$-smooth functions and $f_\pm(0)>0$. Then, it can be easily checked that the homeomorphism
\begin{align}\label{PWShomeo}
 T(x,y)= \begin{cases}
  (x\sqrt{f_-(x)},y), \ x<0,\\ 
  (x,y), \ x=0,\\
  (x\sqrt{f_+(x)},y), \ x>0,
 \end{cases}
  \end{align}
is a $\Sigma$-equivalence in the sense of \cite[Definition 2.20]{GST}, which brings $X_-$ (resp. $X_+$) defined in \eqref{PWSLienard}, locally near $(x,y)=(0,0)$, into $\widetilde X_-$ (resp. $\widetilde X_+$), after multiplication by $\psi'_->0$ (resp. $\psi'_+>0$), where  
\begin{align}\label{PWSLienardNormal}
 \widetilde X_-: \begin{cases}
  \dot x=y-x^2 ,\\ 
  \dot y=\epsilon G_-(\psi_-(x))\psi'_-(x), \end{cases}
     x\le 0,\
 \widetilde X_+: \begin{cases}
  \dot x=y-x^2 ,\\ 
  \dot y=\epsilon G_+(\psi_+(x))\psi'_+(x), \end{cases} x\ge 0,
  \end{align}
  and $\psi_\pm$ is the inverse of $x\to x\sqrt{f_\pm(x)}$. System \eqref{PWSLienardNormal} has a Li\'enard form similar to \eqref{PWSLienard}, with a PWS slow-fast Hopf point at $(x,y)=(0,0)$. We show that \eqref{PWSLienard} and \eqref{PWSLienardNormal} have the same slow relation function (Section \ref{proofThm1}), and therefore they produce the same orbits. Thus, it suffices to give a complete fractal classification of \eqref{PWSLienardNormal}, using the Minkowski dimension. Such fractal classification is given in Theorem \ref{thm2} below. In summary, Theorem \ref{thm1} follows from Theorem \ref{thm2}.
  
 We define
  \begin{equation}\label{FunctionMultip}
  \bar{G}(x):=G_-(\psi_-(x))\psi'_-(x)+G_+(\psi_+(-x))\psi'_+(-x).     
  \end{equation}
  
In what follows, $m_0(\bar G)$ denotes the multiplicity of the zero $x=0$ of $\bar G$, and $g_{m_0(\bar G)}\ne 0$ denotes the $m_0(\bar G)$-th Taylor coefficient of $\bar G$ about $x=0$, if $m_0(\bar G)$ is finite. In Section \ref{proofThm1} we prove the following result.

  \begin{theorem}\label{thm2} Consider \eqref{PWSLienardNormal} and let $H$ be its slow relation function. Given $y_0\in ]0,y^*[$, then the following statements hold.
  \begin{enumerate}
      \item  Suppose that $1\le m_0(\bar G)<\infty$. If $(-1)^{m_0(\bar G)+1}g_{m_0(\bar G)}>0$ (resp. $<0$), then the orbit $U=\{y_l \ | \ l\in\mathbb N\}$ of $y_0$ by $H$ (resp. $H^{-1}$) tends monotonically to $0$ and holds the bijective correspondence
\begin{equation}
      \label{1-1}
    m_0(\bar G)=\frac{1+\dim_BU}{1-\dim_B U}.  
  \end{equation}
  Moreover, if $1< m_0(\bar G)<\infty$, then $U$ is Minkowski nondegenerate.
  \item  If $m_0(\bar G)=\infty$, then $\dim_BU=1$.
  \end{enumerate}
  The above results do not
depend on the choice of the initial point $y_0\in ]0,y^*[$.
  \end{theorem}

  \begin{remark}
      \label{remark-normal}
      From \eqref{1-1}, it follows that
      $$\dim_BU=\frac{m_0(\bar G)-1}{m_0(\bar G)+1}.$$
      This and Statement 2 of Theorem \ref{thm2} imply \eqref{formula-BOX}.

      If we assume that $F$ and $G$ in \eqref{PWSLienard-Intro-FG} are $C^\infty$-smooth, then the function $\bar G$ defined in \eqref{FunctionMultip} is even and, as a consequence of Theorem \ref{thm2}, we have the following sequence of values of $\dim_BU$: $1, \frac{1}{3},\frac{3}{5},\frac{5}{7},\dots$.
  \end{remark}

\subsection{Fractal analysis of bounded balanced canard cycles}\label{sectionbounded}

In this section we focus on the fractal analysis of bounded balanced canard cycles $\Gamma_{\hat y}$, with $\hat y>0$. We call $\Gamma_{\hat y}$ a balanced canard cycle if $I(\hat y)=0$, with $I$ being the integral defined in \eqref{SDI-total}. From now on, we assume that $\Gamma_{\hat y}$ is balanced. Due to Equations \eqref{assum2}, \eqref{PWSSDI} and the Implicit Function Theorem, it follows that there is a function $H(y)$ such that $H(\hat y)=\hat y$ and 
$$I_-(H(y))=I_+(y),$$
for $y$ kept close to $\hat y$ (see \eqref{slow-relation-def}). If we differentiate this last equation, we obtain $H'>0$ (recall that $I_\pm'(y)<0$). This implies that orbits generated by $H$ are monotone.

Assume that there exists a small $y^*>0$ such that $I$ is nonzero in the open interval $]\hat y,\hat y+y^*[$. Given $y_0\in ]\hat y,\hat y+y^*[$, if $I>0$ (resp. $I<0$) in $]\hat y,\hat y+y^*[$, then $U$ denotes the orbit $\{y_0,y_1,y_2,\dots\}$ defined by
\begin{center}
$I_-(y_{l+1})=I_+(y_l)$ \ (resp. $I_-(y_{l})=I_+(y_{l+1})$), \ with $l\ge 0$.
\end{center}

Notice that $U$ is the orbit of $y_0$ by $H$ (resp. $H^{-1}$) and it converges monotonically to $\hat y$ as $l\to\infty$. See also Section \ref{sectionHopf}.

Denote by $m_{\hat y}(I)$ the multiplicity of the zero $y=\hat y$ of $I$. When $m_{\hat y}(I)$ is finite, $i_{m_{\hat y}(I)}\ne 0$ is the $m_{\hat y}(I)$-th Taylor coefficient of $I$ about $y=\hat y$. The fractal classification of $\Gamma_{\hat{y}}$ is given in Theorem \ref{thm3} and it will be proved in Section \ref{proofThm3}.

\begin{theorem}
    \label{thm3}
Consider system \eqref{PWSLienard} and let $H$ be its slow relation function defined near a balanced canard cycle $\Gamma_{\hat y}$. If $y_0\in ]\hat y,\hat y+y^*[$, then the following statements are true.
  \begin{enumerate}
      \item  Suppose that $1\le m_{\hat y}(I) <\infty$. If $i_{m_{\hat y}(I)}>0$ (resp. $<0$), then the orbit $U=\{y_l \ | \ l\in\mathbb N\}$ of $y_0$ by $H$ (resp. $H^{-1}$)  tends monotonically to $\hat y$ and the following bijective correspondence holds
\begin{equation}
      \label{1-1-balanced-PWS}
    m_{\hat y}(I)=\frac{1}{1-\dim_B U}.  
  \end{equation}
  Moreover, if $1< m_{\hat y}(I)<\infty$, then $U$ is Minkowski nondegenerate.
  \item  If $m_{\hat y}(I)=\infty$, then $\dim_BU=1$.
  \end{enumerate}
  The above results do not
depend on the choice of the initial point $y_0\in ]\hat y,\hat y+y^*[$.
  \end{theorem}
  
\begin{remark}\label{remark-balanceddd}
Using \eqref{1-1-balanced-PWS}, it follows that
    $$\dim_B U=\frac{m_{\hat y}(I)-1}{m_{\hat y}(I)}.$$

This result, combined with Statement 2 of Theorem \ref{thm3}, produces the following set of values of $\dim_B U$: $0,\frac{1}{2},\frac{2}{3},\dots,1$. We highlight that, in the PWS setting, we do not obtain new values of the Minkowski dimension in comparison with the smooth setting due to the $C^\infty$-smoothness of $I_-$ and $I_+$ at $y=\hat y$. In other words, the above values also can be produced by balanced canard cycles in the smooth setting (see also \cite{BoxRenato}).
\end{remark}

\subsection{Fractal analysis of PWS classical Li\'enard equations near infinity}\label{sectionunbounded}

In this section we consider a special case of \eqref{PWSLienard}, namely PWS classical  Li\'{e}nard equations of degree $n+1$, which are given by
 \begin{equation}
\label{PWSLienard_inf}
    \begin{vf}
        \dot{x} &=& y-\sum\limits_{k=2}^{n+1} B_k^-x^k, \\
        \dot{y} &=&-\epsilon A_1^{-}x,
    \end{vf} \ \ x\le 0, \ \ 
    \begin{vf}
        \dot{x} &=& y-\sum\limits_{k=2}^{n+1} B_k^+x^k, \\
        \dot{y} &=&-\epsilon A_1^{+}x,
    \end{vf} \ \ x\ge 0,
\end{equation}
where $n \in\mathbb N_1$, $B_{n+1}^\pm\ne 0$, $A_{1}^{\pm}> 0$ and $B_{2}^{\pm}>0$. It is clear that system \eqref{PWSLienard_inf} satisfies \eqref{assum1}. Additionally, it is assumed that \eqref{PWSLienard_inf} satisfies \eqref{assum2} with $L_-=]-\infty,0[$ and $L_+=]0,\infty[$. As a direct consequence of \eqref{assum2}, we have $(-1)^{n+1}B_{n+1}^->0$ and $B_{n+1}^+>0$. 

After a rescaling $(x,t)=(a^{\pm}X,c^{\pm}T)$, we can bring \eqref{PWSLienard_inf} into
\begin{equation}
\label{model-Lienard1}
    \begin{vf}
        \dot{x} &=& y- \left((-1)^{n+1}x^{n+1}+\sum\limits_{k=2}^{n} b_k^-x^k \right)\\
        \dot{y} &=&-\epsilon a_1^{-}x
    \end{vf}
  ´
    \begin{vf}
        \dot{x} &=& y -\left(x^{n+1}+\sum\limits_{k=2}^{n} b_k^+x^k  \right)\\
        \dot{y} &=&  -\epsilon a_{1}^{+}x
    \end{vf}
\end{equation}
where $a_{1}^{\pm}> 0$, $b_{2}^{\pm}>0$ (when $n=1$, we have $b_{2}^{\pm}=1$) and we denote $(X,T)$ again by $(x,t)$. It can be easily checked that \eqref{PWSLienard_inf} and \eqref{model-Lienard1} have the same slow relation function, and therefore it suffices to give a complete fractal classification of \eqref{model-Lienard1} near infinity.

Denote $F_-(x)=(-1)^{n+1}x^{n+1}+\sum\limits_{k=2}^{n} b_k^-x^k$, $F_{+}(x)=x^{n+1}+\sum\limits_{k=2}^{n} b_k^+x^k$ and $G_\pm(x)=- a_{1}^{\pm}x$. Then we have 
$$\frac{F'_\pm(x)^2}{G_\pm(x)}=-\frac{(n+1)^2}{a_1^\pm}x^{2n-1}\left(1+o(1)\right), \ x\to\pm\infty.$$

This and \eqref{PWSSDI} imply that $I_\pm(y)\to -\infty$ as $y\to +\infty$. Let us recall that $I_\pm$ are (strictly) decreasing functions and $I_\pm(y)\to 0$ as $y$ tends to $0$. We conclude that the slow relation function $H$ (see \eqref{slow-relation-def}) is well-defined for all $y>0$.

We suppose that there exists $y^*>0$ large enough such that $I=I_+-I_-$ is nonzero in the open interval $]y^*,\infty[$. Given $y_0\in ]y^*,\infty[$, if $I<0$ (resp. $I>0$) on $]y^*,\infty[$, we define the orbit $U=\{y_0,y_1,y_2,\dots\}$ by
\begin{center}
$I_-(y_{l+1})=I_+(y_l)$ \ (resp. $I_{-} (y_{l})=I_+(y_{l+1})$), \ with $l\ge 0$.
\end{center}
Then $U$ is the orbit of $y_0$ by $H$ (resp. $H^{-1}$) and it tends monotonically to $+\infty$. We define the lower Minkowski dimension of $U$ by 
\begin{equation}
    \label{Minkowski-def-inf}
    \underline\dim_BU=\underline\dim_B \Big\{\frac{1}{y_l^\frac{1}{n+1}} \ | \ l\in \mathbb{N} \Big\},
\end{equation}
and similarly for the upper Minkowski dimension $\overline\dim_BU$. If $\underline\dim_BU=\overline\dim_BU$,
then we denote it by $\dim_BU$ and call it the Minkowski dimension of $U$. In Section \ref{proofThm4}, system \eqref{model-Lienard1} will be studied on the Poincar\'e--Lyapunov
disc of degree $(1, n + 1)$ and the exponent of $y_l$ in \eqref{Minkowski-def-inf} is related to that degree (see also \cite[Section 2]{MinkLienard} for more details). 

We introduce the notation
\begin{align}
    F_{+}(x)-F_{-}(-x)&= \sum\limits_{k=2}^{n} \left(b_k^++(-1)^{k+1}b_k^-\right)x^k \nonumber \\ 
    &=(b_{2}^{+}-b_{2}^{-})x^2+(b_{3}^{+}+b_{3}^{-})x^3+\dots+(b_{n}^{+}+(-1)^{n+1}b_{n}^{-})x^n \nonumber\\
    &=:f_{2}x^2+f_{3}x^3+\dots+f_{n}x^{n}.  \label{Function_f_n}
\end{align}

In the case $n=1$, then $F_{-}(x)-F_{+}(-x)=0$. When one of the coefficients $f_k$ in \eqref{Function_f_n} is nonzero, we denote by $k_0$ the maximal $k$ with the property $f_k\ne 0$. Now we are able to state the main result of this section. Theorem \ref{thm4} is proved in Section \ref{proofThm4}.

\begin{theorem}\label{thm4}
     Consider system \eqref{model-Lienard1} and the associated slow relation function $H$. Given $y_0\in ]y^*,\infty[$, the following statements hold. 
    \begin{enumerate}
        \item Suppose that $a_{1}^{-}\neq a_{1}^{+}$. If $a_{1}^{-}> a_{1}^{+}$ (resp. $a_{1}^{-}< a_{1}^{+}$), then the orbit $U=\{y_l \ | \ l\in\mathbb N\}$ of $y_0$ by $H$ (resp. $H^{-1}$)  tends monotonically to $+\infty$ and 
        \begin{equation*}
            \dim_{B}U=0.
        \end{equation*}
        \item \textbf Suppose that $a_{1}^{-}= a_{1}^{+}$ and that one of the coefficients $f_k$ is nonzero with $k_0\ne n-1$. If $f_{k_0}(1+k_0-n)>0$ (resp. $f_{k_0}(1+k_0-n)<0$), then the orbit $U=\{y_l \ | \ l\in\mathbb N\}$ of $y_0$ by $H$ (resp. $H^{-1}$) tends monotonically to $+\infty$,  $U$ is Minkowski nondegenerate and
        \begin{equation} \label{formula-BOX_inf}
            \dim_{B} U=\frac{n+1-k_0}{n+2-k_0}.
        \end{equation}
    \end{enumerate}
    These results do not
depend on the choice of the initial point $y_0$.
\end{theorem} 

If $a_{1}^{-}= a_{1}^{+}$ and $f_2=\dots=f_n=0$, then  we have $H(y)=y$ ($I\equiv 0$) and each orbit $U$ is a fixed point of $H$ with the trivial Minkowski dimension ($\dim_{B}U=0$).

\begin{remark}
Suppose that $F_-=F_+$ and it is a polynomial of even degree $n+1$ ($n$ odd) and $G_-=G_+$ ($a_{1}^{-}= a_{1}^{+}$). Then \eqref{Function_f_n} implies that $f_k=0$ for $k$ even (hence, $k_0$ is odd and $k_0\ne n-1$). From \eqref{formula-BOX_inf} it follows that we have the following Minkowski dimensions of $U$: $\frac{1}{2},\frac{3}{4},\dots,\frac{n-4}{n-3},\frac{n-2}{n-1}$ (see also \cite[Remark 1]{MinkLienard}).

In the PWS setting, we get for $n$ odd or even: $0,\frac{1}{2},\frac{3}{4},\frac{4}{5},\dots,\frac{n-2}{n-1},\frac{n-1}{n}$. We refer to Statement 1 and Statement 2 of Theorem \ref{thm4}. The case $k_0= n-1$ is a topic of further study. We believe that in this case the Minkowski dimension is different from $\frac{2}{3}$ and we have to deal with more complicated expansions of functions. 
\end{remark}

\section{Proof of Theorems \ref{thm1}--\ref{thm4}}\label{section-proof-all}
\subsection{Proof of Theorems \ref{thm1} and \ref{thm2}}\label{proofThm1}

We consider \eqref{PWSLienard} and assume that \eqref{assum1} is satisfied. Recall that the integrals $I_-$ and $I_+$ are defined in \eqref{PWSSDI} and $I$ in \eqref{SDI-total}. We denote by $\widetilde I_-,\widetilde I_+,\widetilde I$ the integrals $I_-,I_+,I$ computed for system \eqref{PWSLienardNormal}: 
\begin{equation}\label{PWSSDINormal1}
\widetilde I_\pm(y)=-4\int_{\pm\sqrt y}^{0}\frac{x^2}{G_\pm(\psi_\pm(x))\psi'_\pm(x)}dx, \ \widetilde I(y)=\widetilde I_+(y)-\widetilde I_-(y).
\end{equation}

Since the slow divergence integral is invariant under changes of coordinates and time reparameterizations (see \cite[Chapter 5]{DDR-book-SF}), we have $\widetilde I_\pm (y)=I_\pm(y)$ (this can be easily seen if we use the change of variable $s=\psi_\pm(x)$ in $\widetilde I_\pm$). Now, it is clear that $H$ defined in \eqref{slow-relation-def} is also the slow relation function of \eqref{PWSLienardNormal} ($\widetilde I_-(H(y))=\widetilde I_+(y)$). This implies that Theorem \ref{thm1} follows from Theorem \ref{thm2}. Therefore, in the rest of this section we prove Theorem \ref{thm2} and, from now on, When we refer to \eqref{PWSSDINormal1}, we use $I_-,I_+,I$ instead of $\widetilde I_-,\widetilde I_+,\widetilde I$.

We will prove Theorem \ref{thm2} for $I>0$ on $]0,y^*[$. Let $y_0\in ]0,y^*[$. Then the orbit $U$ of $y_0$ by $H$ tends monotonically to $0$ as $l\to \infty$ ($I_-(y_{l+1})=I_+(y_l)$ with $l\ge 0$, see Section \ref{sectionHopf}). The case where $I<0$ on $]0,y^*[$ can be
treated in a similar fashion.

We introduce the notation 
$$g_\pm(x)=G_\pm(\psi_\pm(x))\psi'_\pm(x),$$
and therefore Equation \eqref{FunctionMultip} can be written as 
$$\bar{G}(x) = g_{-}(x) + g_{+}(-x).$$

From assumption \eqref{assum1} and the definition of $\psi_\pm$ after \eqref{PWSLienardNormal}, it follows that $g_\pm(0)=0$ and $g'_\pm(0)<0$. Moreover, using Taylor series of $g_{\pm}$ at the origin, one can write the integrand in \eqref{PWSSDINormal1} as 
\begin{align}\label{integrand1}
    \frac{x^2}{g_\pm(x)}&=\frac{1}{g'_\pm(0)}x+O(x^2).
\end{align}

For the integral $I(y_l)$ we get  \begin{align}\label{intEE1}
  I(y_l)=& I_+(y_l)-I_-(y_{l+1})+I_-(y_{l+1})-I_-(y_{l})\nonumber \\
  =&4\int_{-\sqrt {y_l}}^{-\sqrt {y_{l+1}}}\frac{x^2}{g_-(x)}dx\nonumber\\
  =& -\frac{2}{g_-'(0)}y_l\int_{\frac{y_{l+1}}{y_l}}^{1}\left (1+O(\sqrt{y_l s})\right)ds.
 \end{align}

In the second step in \eqref{intEE1} we use $I_-(y_{l+1})=I_+(y_l)$ and in the last step we use \eqref{integrand1} and then the change of variable $s=\frac{x^2}{y_l}$.
 
The idea is to compare \eqref{intEE1} with equivalent asymptotic expansions of \eqref{PWSSDINormal1}. There are three cases that must be considered: $m_0(\bar G)=1$, $1< m_0(\bar G)<\infty$ and $m_0(\bar G)=\infty$. \\
\\
(a) $m_0(\bar G)=1$. This holds if, and only if, $g'_-(0)\ne g'_+(0)$. 

Using \eqref{PWSSDINormal1} and \eqref{integrand1}, we get 
 \begin{align}\label{intasymp1}
  I(y)=&-4\int_{\sqrt y}^{0}\frac{x^2}{g_+(x)}dx+4\int_{-\sqrt y}^{0}\frac{x^2}{g_-(x)}dx\nonumber\\
  =& 2\left(\frac{1}{g_+'(0)}-\frac{1}{g_-'(0)}\right)y(1+o(1)),
 \end{align}
where $o(1)$ tends to zero as $y\to 0$. It is clear from \eqref{intasymp1} that $I>0$ is equivalent to $g_-'(0)>g_+'(0)$ (and recall that $g'_\pm(0)<0$). From \eqref{intasymp1} with $y=y_l$, \eqref{intEE1} and $y_l\to 0$ as $l\to\infty$ it follows that
 $$\lim_{l\to\infty}\frac{y_{l+1}}{y_l}=\frac{g_-'(0)}{g_+'(0)}\in ]0,1[.$$
 
This implies that the orbit $U$ converges exponentially to zero as $l\to \infty$, that is, there exist $\lambda\in]0,1[$ and a constant $c>0$ such that $0<y_l\le c\lambda^l$ for all $l$. It follows from \cite[Lemma 1]{EZZ} that $\dim_BU=0$ and the Minkowski dimension does not depend on the choice of $y_0$. This completes the proof of Statement 1 of Theorem \ref{thm2} when $m_0(\bar G)=1$.
\\
\\
(b) $1< m_0(\bar G)<\infty$. This holds if, and only if, $g'_-(0)= g'_+(0)$. Using Taylor series at $0$, one can write the integrand of \eqref{PWSSDINormal1} as
\begin{align}\label{integrand2}
    \frac{x^2}{g_\pm(x)}&=\frac{x^2}{T_\pm(x)+\gamma_\pm x^{m_0(\bar G)}+O(x^{m_0(\bar G)+1})} \nonumber\\
    &=\frac{x^2}{T_\pm(x)}-\frac{\gamma_\pm}{\left(g'_\pm(0)\right)^2} x^{m_0(\bar G)}+O(x^{m_0(\bar G)+1}),
\end{align}
where $T_\pm$ is the Taylor polynomial of degree $m_0(\bar G)-1$ of $g_\pm$ at $x=0$, $T_\pm'(0)=g'_\pm(0)$ and $\gamma_\pm$ is the $m_0(\bar G)$-th Taylor coefficient of $g_\pm$ about $x=0$. Notice that 
\begin{equation}\label{gamma+-}
    g_{m_0(\bar G)}=\gamma_-+(-1)^{m_0(\bar G)}\gamma_+\ne 0,
\end{equation}
where $g_{m_0(\bar G)}$ denotes the $m_0(\bar G)$-th Taylor coefficient of $\bar G$ about $x=0$ (recall its definition after Equation \eqref{FunctionMultip}). A direct consequence of \eqref{gamma+-} is 
\begin{equation}\label{eq-gamma+-further}
(-1)^{m_0(\bar G)+1}g_{m_0(\bar G)} = (-1)^{m_0(\bar G)+1}\gamma_{-} - \gamma_{+}. 
\end{equation}

The integral $I$ can be written as
 \begin{align}\label{intasymp2}
  I(y)=&-4\int_{\sqrt y}^{0}\frac{x^2}{T_+(x)}dx+4\int_{-\sqrt y}^{0}\frac{x^2}{T_-(x)}dx\nonumber\\
  \ & +\frac{4(-1)^{m_0(\bar G)+1}g_{m_0(\bar G)}}{\left(g_+'(0)\right)^2(m_0(\bar G)+1)}y^\frac{m_0(\bar G)+1}{2}(1+o(1))\nonumber\\
  =& \frac{4(-1)^{m_0(\bar G)+1}g_{m_0(\bar G)}}{\left(g_+'(0)\right)^2(m_0(\bar G)+1)}y^\frac{m_0(\bar G)+1}{2}(1+o(1)),
 \end{align} where $o(1)$ tends to zero as $y\to 0$. In the first step in \eqref{intasymp2} we use \eqref{PWSSDINormal1}, \eqref{integrand2}, \eqref{gamma+-} and the fact that $g'_-(0)= g'_+(0)$ (because $m_0(\bar G)>1$). We also used the relation \eqref{eq-gamma+-further}. Moreover, in the second step we use the fact that
 $$\int_{\sqrt y}^{0}\frac{x^2}{T_+(x)}dx=\int_{-\sqrt y}^{0}\frac{x^2}{T_-(x)}dx.$$ 

This follows directly from $T_-(x)=-T_+(-x)$ (observe that the Taylor polynomial $T_-(x)+T_+(-x)$ of degree $m_0(\bar G)-1$ of $\bar G$, defined in \eqref{FunctionMultip}, at $x=0$ is identically zero). A simple consequence of \eqref{intasymp2} is that $I>0$ if, and only if, $(-1)^{m_0(\bar G)+1}g_{m_0(\bar G)}>0$.

Finally, \eqref{intasymp2} and The Mean Value Theorem for Integrals applied to \eqref{intEE1} yield
\begin{equation}\label{eqsim}
 y_l-y_{l+1}\simeq y_l^\frac{m_0(\bar G)+1}{2}, l\to\infty,
\end{equation}
where the notation $\simeq$ was introduced in Section \ref{Minkowski-def-separated}. Note that $\nu:=\frac{m_0(\bar G)+1}{2}>1$. It follows from Equation \eqref{eqsim} and \cite[Theorem 1 and Remark 1]{EZZ} that the orbit $U$ is Minkowski nondegenerate and
$$\dim_BU=1-\frac{1}{\nu}=\frac{m_0(\bar G)-1}{m_0(\bar G)+1},$$
and these results are independent of the choice of $y_0$. We have proved Statement 1 of Theorem \ref{thm2} for $1< m_0(\bar G)<\infty$.\\
\\
(c) $m_0(\bar G)=\infty$. Using similar steps to the ones used in case (b), it is not difficult to see that for every $\widetilde \nu>0$ then $I(y)=O(y^{\widetilde \nu})$, when $y\to 0$. This and \eqref{intEE1} imply that for every $\widetilde \nu>0$ then $y_l-y_{l+1}=O(y_l^{\widetilde \nu})$, when $l\to\infty$. Following \cite[Theorem 6]{EZZ} we have $\dim_BU=1$, and once again the Minkowski dimension does not depend on $y_0$. This completes the proof of Statement 2 of Theorem \ref{thm2}.

\subsection{Proof of Theorem \ref{thm3}}\label{proofThm3}

Suppose that $\Gamma_{\hat y}$ is a balanced canard cycle with the associated slow relation function $H$ defined in Section \ref{sectionbounded}. We prove Theorem \ref{thm3} for $I>0$ in the open interval $ ]\hat y,\hat y+y^*[$. In this case, the orbit $U$ of $y_0\in ]\hat y,\hat y+y^*[$ by $H$ converges monotonically to $\hat y$ as $l\to \infty$. We have $\hat y < H(y) < y$, for each $y\in ]\hat y,\hat y+y^*[$. The case $I<0$ on $ ]\hat y,\hat y+y^*[$ can be treated in a similar way. 

We have
\begin{align*}
     I(y)&=I_+(y)-I_-(y)\\
     &=I_+(y)-I_-(H(y))+I_-(H(y))-I_-(y)\\
     &=\int_{\alpha(y)}^{\alpha(H(y))}\frac{F'_-(x)^2}{G_-(x)}dx.
     \end{align*}
     
In the last step, we use $I_-(H(y))=I_+(y)$ and \eqref{PWSSDI}. From \eqref{assum2} and The Fundamental Theorem of Calculus, it follows that $\alpha'(y)<0$ and 
$$I(y)=\zeta(y)(y-H(y)),$$
where $y$ is kept near $\hat y$ and $\zeta$ is a positive smooth function.
This implies that $y=\hat y$ is a zero of multiplicity $m$ of $I(y)$ if, and only if, $y=\hat y$ is a zero of multiplicity $m$ of $y-H(y)$.

We know that the following result holds (see \cite{EZZ,MRZ}).
\begin{theorem}\label{th-EZZ}
 Let $\tilde F$ be a smooth function on $[0,\tilde y[$, $\tilde F(0)=0$ and $0<\tilde F(y)<y$ for each $y\in ]0,\tilde y[$. Define $\tilde H=id-\tilde F$ and let $U$ be the orbit of $y_0\in ]0,\tilde y[$ by $\tilde H$. Then $\dim_BU$ is independent of the initial point $y_0$ and, for $1\le m\le \infty$, the following bijective correspondence holds:
 \begin{equation}
     m=\frac{1}{1-\dim_BU}\nonumber
 \end{equation}
 where $m$ is the multiplicity of the zero $y=0$ of $\tilde F$. If $m=\infty$, then $\dim_BU=1$.
\end{theorem}

Theorem \ref{thm3} follows from Theorem \ref{th-EZZ} with $\tilde y=y^*$, $\tilde F(y)=y+\hat y-H(y+\hat y)$ and $\tilde H(y)=H(y+\hat y)-\hat y$. When $1\le m_{\hat y}(I) <\infty$, it is clear that $I>0$ if, and only if  $i_{m_{\hat y}(I)}>0$. Moreover, for $1< m_{\hat y}(I)<\infty$, the orbit $U$ is Minkowski nondegenerate (see e.g. \cite[Theorem 1]{EZZ}).

\subsection{Proof of Theorem \ref{thm4}}\label{proofThm4}
In this section we focus on system \eqref{model-Lienard1} and prove Theorem \ref{thm4}. In Section \ref{m<-section-infty}, we study the Poincar\'{e}--Lyapunov compactification of \eqref{model-Lienard1} and detect the unbounded canard cycle. Section \ref{subsubsec-analysis} is devoted to a fractal classification of \eqref{model-Lienard1} at infinity in the positive $y$-direction. In Section \ref{complete-proof}, we complete the proof of Theorem \ref{thm4} by using the invariance of the slow divergence integral under changes of coordinates and changes of time.

\subsubsection{Poincar\'{e}--Lyapunov compactification}\label{m<-section-infty}

This section is devoted to study the dynamics of \eqref{model-Lienard1} near infinity on the Poincar\'{e}--Lyapunov disc of degree $(1,n+1)$.

In the positive $x$-direction we use the transformation 
\begin{equation}
x=\frac{1}{r}, \ y=\frac{\bar y}{r^{n+1}}, \nonumber
\end{equation}
with $r>0$ small and $\bar y$ kept in a large compact set.  In the new coordinates $(r,\bar y)$, system \eqref{model-Lienard1} becomes (after multiplication by $r^{n}$)
\begin{equation}
\label{model-infinity m^+< pos x}
    \begin{vf}
        \dot{r} & = & -r\left(\bar y-1-\sum\limits_{k=2}^{n}b_k^+ r^{n+1-k} \right),   \\
        \dot{\bar y} & = &-\epsilon a_{1}^{+} r^{2n} -(n+1)\bar y \left(\bar y-1-\sum\limits_{k=2}^{n}b_k^+ r^{n+1-k} \right).
    \end{vf}
\end{equation}

 Suppose that $\epsilon=0$. On the line $\{r=0\}$ system \eqref{model-infinity m^+< pos x} has two singular points: $\bar y=0$ and $\bar y=1$. The eigenvalues of the linear part at $\bar y=0$ are  given by $(1,n+1)$ and the eigenvalues of $\bar y=1$ are given by $(0,-(n+1))$. This implies that $\bar y=0$ is a hyperbolic repelling node and $\bar y=1$ is a semi-hyperbolic singularity with the $\bar y$-axis as the stable manifold and the curve of singular points 
$\bar y=1+\sum_{k=2}^{n}b_k^+ r^{n+1-k}$
as center manifold. 

It is not difficult to see, using the invariance under the flow and asymptotic expansions in $\epsilon$, that center manifolds of \eqref{model-infinity m^+< pos x}$+0\frac{\partial}{\partial \epsilon}$ at $(r,\bar y, \epsilon)=(0,1,0)$ are given by 
\[\bar y=1+\sum_{k=2}^{n} b_k^+ r^{n+1-k}-r^{2n}\left(\frac{a_{1}^{+}}{n+1}+O(r)\right)\epsilon+O(\epsilon^2).\]

If we substitute this for $\bar y$ in the first component of \eqref{model-infinity m^+< pos x}, divide out $\epsilon$ and let $\epsilon\to 0$, we obtain the slow dynamics 
\[r'=r^{2n+1}\left(\frac{a_{1}^{+}}{n+1} +O(r)\right).\]

Since $a_{1}^{+}>0$, the slow dynamics points away from the singularity $r=0$.

In the negative $x$-direction we use the transformation
\begin{equation}
x=\frac{-1}{r}, \ y=\frac{\bar y}{r^{n+1}},\nonumber
\end{equation}
and system \eqref{model-Lienard1} changes (after multiplication by $r^{n}$) into
\begin{equation}
\label{model-infinity m^-< neg x}
    \begin{vf}
        \dot{r} &=& r\left(\bar y - 1 -\sum\limits_{k=2}^{n}b_k^-(-1)^k r^{n+1-k} \right),   \\
        \dot{\bar y} &=&\epsilon a_{1}^{-} r^{2n} +(n+1)\bar y \left(\bar y - 1 - \sum\limits_{k=2}^{n}b_k^- (-1)^k r^{n+1-k} \right) .
    \end{vf}
\end{equation} 

 When $\epsilon=0$, the line $\{r=0\}$ contains two singularities of \eqref{model-infinity m^-< neg x}: $\bar y = 0$ and $\bar y = 1$. The eigenvalues of the linear part at $\bar y=0$ (resp. $\bar y = 1$) are $(-1, -(n+1))$ (resp. $(0,n+1)$). Hence, $\bar y=0$ is a hyperbolic attracting node and $\bar y = 1$ is a semi-hyperbolic singularity with the $\bar y$-axis as the unstable manifold. The curve of singularities is given by $\bar y = 1 + \sum_{k=2}^{n}b_k^-(-1)^k r^{n+1-k}$.
 The slow dynamics along the curve of singularities is given by 
\[r'=r^{2n+1}\left(\frac{-a_{1}^{-}}{n+1} +O(r)\right),\]
and it points towards
the origin $r=0$ because $a_{1}^{-}>0$.

In the positive and negative $y$-direction, there are no extra singularities. After putting all the information together, we get the phase portrait near infinity of \eqref{model-Lienard1} for $\epsilon=0$, including direction of the slow dynamics (see Figure \ref{fig_PWS}). Now, it is clear that the unbounded canard cycle consists of the curve of singularities of \eqref{model-Lienard1}, denoted by $S$, and the regular orbit connecting the two semi-hyperbolic singularities at infinity (these two points are the end points of $S$). 

\begin{figure}[htb]
	\begin{center}
		\includegraphics[scale=0.5]{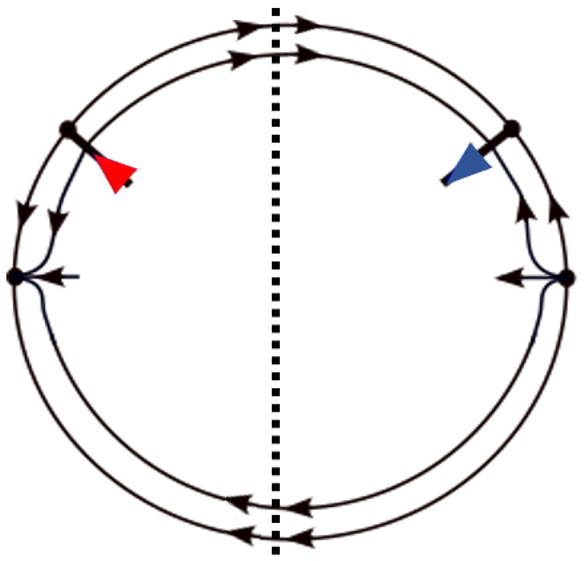}
        {\footnotesize
        \put(-87,-6){$x=0$}
             }
         \end{center}
 \caption{The phase portrait near infinity of \eqref{model-Lienard1}, with $\epsilon=0$. We have a crossing near the switching line $x=0$.}\label{fig_PWS}
\end{figure}

\subsubsection{Fractal Analysis Near Infinity}\label{subsubsec-analysis}

Since the attracting branch and the repelling branch of the curve of singularities
are visible in the positive $y$-direction (see Figure \ref{fig_PWS}), it is natural to present the fractal analysis in the positive $y$-direction. We have
\begin{equation}
x=\frac{\bar x}{r}, \ y=\frac{1}{r^{n+1}},\nonumber
\end{equation}
and system \eqref{model-Lienard1} becomes (after multiplication by $r^{n}$)

\begin{equation} \label{PL_negative}
    \begin{cases}
        \dot{r} =\frac{\epsilon a_{1}^{-}}{n+1} r^{2n+1}\bar{x},\\
        \dot{\bar{x}}  = 1-\left((-1)^{n+1}\bar{x}^{n+1}+\sum\limits_{k=2}^{n} b_k^- r^{n+1-k} \bar{x}^k \right) +\frac{\epsilon a_{1}^{-}}{n+1} r^{2n} \bar{x}^2,
    \end{cases} \ \ \bar x\le 0,
\end{equation}

\begin{equation} \label{PL_positive}
    \begin{cases}
        \dot{r} = \frac{\epsilon a_{1}^{+}}{n+1}r^{2n+1}\bar{x}, \\
        \dot{\bar{x}} = 1-\left(\bar{x}^{n+1}+ \sum\limits_{k=2}^{n} b_k^+r^{n+1-k}\bar{x}^k \right)+\frac{\epsilon a_{1}^{+}}{n+1} r^{2n}\bar{x}^2,
    \end{cases} \ \ \bar x\ge 0.
\end{equation}

When $\epsilon=0$, system (\ref{PL_negative}) has a normally repelling curve of singularities $\Bar{x}=\Phi_{-}(r)$ satisfying $\Phi_{-}(0)=-1$ and  
\begin{equation}\label{phidef-}
     1-(-1)^{n+1}\Phi_{-}(r)^{n+1}-\sum_{k=2}^{n} b_k^- r^{n+1-k}\Phi_{-}(r)^k=0,
\end{equation}
and (\ref{PL_positive}) has a normally attracting curve of singularities $\Bar{x}=\Phi_{+}(r)$ satisfying $\Phi_{+}(0)=1$ and
\begin{equation}\label{phidef+}
    1-\Phi_{+}(r)^{n+1}-\sum_{k=2}^{n} b_k^+r^{n+1-k}\Phi_{+}(r)^k=0.
\end{equation}

Recall from Equation \eqref{Function_f_n} that $f_k=b_k^++(-1)^{k+1}b_k^-$, for $k=2,\dots,n$. The following lemma gives a useful connection between $\Phi_{-}$ and $\Phi_{+}$.
\begin{lemma} \label{lemma-Phi}
The functions $\Phi_{-}(r)$ and $\Phi_{+}(r)$ defined in \eqref{phidef-} (resp. \eqref{phidef+}) satisfy
\begin{equation}\label{expansionPhi_PWS}
    \Phi_{+}(r)=-\Phi_{-}(r) 
    -\displaystyle\frac{1}{n+1} \sum_{k=2}^{n} f_k r^{n+1-k} \left(1+ O(r) \right), \nonumber
\end{equation}
\end{lemma}
\begin{proof}
The lemma can be easily proved using $\Phi_{+}(r)=-\Phi_{-}(r)$ when $f_2=\dots=f_n=0$ (see \eqref{phidef-} and \eqref{phidef+}).   
\end{proof}

If we substitute $\Phi_{-}(r)$ (resp. $\Phi_{+}(r)$) for $\bar x$ in the $r$-component of (\ref{PL_negative}) (resp. (\ref{PL_positive})) and divide out $\epsilon$, then we obtain the slow dynamics along $\bar{x}=\Phi_{\pm}(r)$:
\begin{equation}\label{slow-dyn-infty}
    r'=\frac{a_{1}^{\pm}}{n+1}r^{2n+1}\Phi_{\pm}(r).
\end{equation}

Let $\Tilde{r}>0$ be small and fixed. For $r \in ]0,\Tilde{r}[$, we define the slow divergence integral of (\ref{PL_negative}) along the portion $[r,\Tilde r]$ of $\Bar{x}=\Phi_{-}(r)$
\begin{equation} \label{J_-}
    J_{-}(r)=-(n+1)\int_{r}^{\Tilde{r}}\frac{(n+1)(-1)^{n+1}\Phi_{-}(s)^{n}+\sum\limits_{k=2}^{n} kb_k^{-}s^{n+1-k}\Phi_{-}(s)^{k-1}}{a_{1}^{-} s^{2n+1}\Phi_{-}(s)} ds < 0,
\end{equation}
and the slow divergence integral of (\ref{PL_positive}) along the portion $[r,\Tilde r]$ of $\Bar{x}=\Phi_{+}(r)$
\begin{equation} \label{J_+}
    J_{+}(r)=-(n+1)\int_{r}^{\Tilde{r}}\frac{(n+1)\Phi_{+}(s)^{n}+\sum\limits_{k=2}^{n} kb_k^{+}s^{n+1-k}\Phi_{+}(s)^{k-1}}{a_{1}^{+}s^{2n+1}\Phi_{+}(s)} ds < 0.
\end{equation}

Observe that $J_+$ is the integral of
the divergence of (\ref{PL_positive}) for $\epsilon=0$, computed in singular points $(s,\Phi_{+}(s))$, where the variable of integration is the time variable of the slow dynamics \eqref{slow-dyn-infty}. A similar remark holds for the integral $J_-$ (but it is computed in backward time). From \eqref{J_-} and \eqref{J_+} it follows that $J_{\pm}(r)\to -\infty$ as $r \rightarrow 0$. 

Let $\Tilde{J}\in\mathbb R$ be arbitrary but fixed. For $r_0 \in ]0,\Tilde{r}[$, we suppose that the sequence $(r_l)_{l \in \N}$, defined by
\begin{equation*}
    J_{+}(r_l)-J_{-}(r_{l+1})=\Tilde{J}, \ l\ge 0,
\end{equation*}
tends monotonically to $0$ as $l \rightarrow \infty$. The case where $(r_l)_{l \in \N}$, tending to $0$ as $l \rightarrow \infty$, is defined by $J_{+}(r_{l+1})-J_{-}(r_{l})=\Tilde{J}$ can be treated in a similar way.
\begin{remark}
    In Section \ref{complete-proof} the constant $\Tilde J$ will be equal to the slow divergence integral $-I\left(\frac{1}{\tilde r^{n+1}}\right)$, with $I$ defined in \eqref{SDI-total}.
\end{remark}
We can write
\begin{equation} \label{J-equation}
    J_{+}(r_l)-J_{-}(r_l)=\Tilde{J}+J_{-}(r_{l+1})-J_{-}(r_l).
\end{equation}

Let us first study $J_{+}(r_l)-J_{-}(r_l)$. For simplicity sake, we write $\Phi_{\pm}=\Phi_{\pm}(s)$. From Lemma \ref{lemma-Phi} and the relation $b_{k}^{+} = f_{k} + (-1)^{k}b_{k}^{-}$, one can check that  
\begin{align}\label{pomocno-integrand}
    &\frac{(n+1)\Phi_{+}^{n}+\sum\limits_{k=2}^{n} kb_k^{+} s^{n+1-k}\Phi_{+}^{k-1}}{a_{1}^{+} \Phi_{+}}\nonumber\\
    &=\frac{(n+1)(-\Phi_{-})^{n}+\sum\limits_{k=2}^{n} kb_k^{+} s^{n+1-k}(-\Phi_{-})^{k-1}}{-a_{1}^{+} \Phi_{-}} - \sum_{k=2}^{n} f_k s^{n+1-k} \left(\frac{n-1}{a_{1}^{+}}+{O}(s) \right)\nonumber\\
    &=\frac{-(n+1)(-1)^{n+1}\Phi_{-}^{n} \! - \! \sum\limits_{k=2}^{n} kb_k^{-} s^{n+1-k}\Phi_{-}^{k-1} \! + \! \sum\limits_{k=2}^{n} k f_k s^{n+1-k}(-\Phi_{-})^{k-1} }{-a_{1}^{-}\Phi_{-} +(a_{1}^{-}-a_{1}^{+})\Phi_{-}}\nonumber\\
    & \ \ \ \ - \sum_{k=2}^{n} f_k s^{n+1-k} \left(\frac{n-1}{a_{1}^{+}}+{O}(s) \right) \nonumber\\
    &=\frac{(n+1)(-1)^{n+1}\Phi_{-}^{n}+\sum\limits_{k=2}^{n} kb_k^{-}s^{n+1-k}\Phi_{-}^{k-1}}{a_{1}^{-}\Phi_{-}} \nonumber \\
    & \ \ \ - \sum_{k=2}^{n} f_k s^{n+1-k} \left(\frac{n-1-k}{a_{1}^{+}}+{O}(s) \right)
    +(a_{1}^{-}-a_{1}^{+})\left( \frac{n+1}{a_{1}^{-}a_{1}^{+}}+{O}(s) \right).
\end{align}

Using the integrals (\ref{J_-}), (\ref{J_+}) and considering the integrand \eqref{pomocno-integrand}, it follows that
\begin{align}\label{integral-pomocno}
    J_{+}(r_l)-J_{-}(r_l)&= 
    (n+1)\int_{r_l}^{\Tilde{r}}\frac{1}{s^{2n+1}} \sum_{k=2}^{n} f_k s^{n+1-k} \left(\frac{n-1-k}{a_{1}^{+}}+{O}(s) \right)ds\nonumber \\
    & \ \ \ -(n+1)\int_{r_l}^{\Tilde{r}}\frac{1}{s^{2n+1}}(a_{1}^{-}-a_{1}^{+})\left( \frac{n+1}{a_{1}^{-}a_{1}^{+}}+{O}(s) \right)ds\nonumber\\
    &=\sum_{k=2}^{n}f_k r_l^{1-n-k}\left(\frac{(n+1)(1+k-n)}{a_{1}^{+}(1-n-k)}+o(1)\right)\nonumber\\
    & \ \ \ +(a_{1}^{-}-a_{1}^{+})r_l^{-2n}\left( -\frac{(n+1)^2}{2na_{1}^{-}a_{1}^{+}}+o(1) \right)+\hat J,
\end{align}
where $o(1) \rightarrow 0$ as $r_l \rightarrow 0$ and $\hat{J}$ is a constant independent of $r_l$. 

Now, consider the term $J_{-}(r_{l+1})-J_{-}(r_l)$ in \eqref{J-equation}. Using \eqref{J_-}, we have
\begin{align}
    J_{-}(r_{l+1})-J_{-}(r_l)&=-\frac{(n+1)^2}{a_{1}^{-}}\int_{r_{l+1}}^{r_l} \frac{1}{s^{2n+1}} \left(1+{O}(s) \right) ds\nonumber\\
    &=-\frac{(n+1)^2}{a_{1}^{-}r_l^{2n}}\int_{\frac{r_{l+1}}{r_l}}^{1} \frac{1}{\tilde s^{2n+1}} \left(1+{O}(r_l\tilde s) \right) d\tilde s,
    \label{Integral_left_to_right}
\end{align}
where in the last step we use the change of variable $s=r_l\tilde s$. 

Now, we use \eqref{integral-pomocno} and \eqref{Integral_left_to_right} in the Equation \eqref{J-equation} and then we get

\begin{align}
    -\frac{(n+1)^2}{a_{1}^{-}}&\int_{\frac{r_{l+1}}{r_l}}^{1} \frac{1}{\tilde s^{2n+1}} \left(1+{O}(r_l\tilde s) \right) d\tilde s\nonumber\\
    &=\sum_{k=2}^{n}f_k r_l^{n+1-k}\left(\frac{(n+1)(1+k-n)}{a_{1}^{+}(1-n-k)}+o(1)\right)\nonumber\\
    & \ \ \ +(a_{1}^{-}-a_{1}^{+})\left( -\frac{(n+1)^2}{2na_{1}^{-}a_{1}^{+}}+o(1) \right)+r_l^{2n}(\hat J-\Tilde{J}).
    \label{Recursion-konacno}
\end{align}

We distinguish between two cases: $a_{1}^{-}\ne a_{1}^{+}$ and $a_{1}^{-}=a_{1}^{+}$. Recall that $a_1^\pm>0$.

\paragraph{Case $a_{1}^{-}\ne a_{1}^{+}$.} Since $(r_l)_{l \in \N}$ tends monotonically to $0$ as $l 
 \rightarrow \infty$, it can be easily seen that \eqref{Recursion-konacno} implies 
$$\lim_{l \rightarrow \infty}\frac{r_{l+1}}{r_l}=\left(\frac{a_{1}^{+}}{a_{1}^{-}}\right)^\frac{1}{2n} \in ]0,1[.$$

Since $a_{1}^{-}> a_{1}^{+}$, we conclude that $\dim_B(r_l)_{l \in \N}=0$ and the Minkowski dimension does not depend on the choice of the initial point $r_0 \in ]0,\Tilde{r}[$ (see  Section \ref{proofThm1}). We remark that when the sequence $(r_l)_{l \in \N}$, converging monotonically to $0$ as $l \rightarrow \infty$, is defined by $J_{+}(r_{l+1})-J_{-}(r_{l})=\Tilde{J}$, we have $a_{1}^{-}< a_{1}^{+}$. 

\paragraph{Case $a_{1}^{-}=a_{1}^{+}$.}Now the right-hand side of \eqref{Recursion-konacno} tends to 0 as $l\to \infty$, and we get
\begin{equation}\label{quotient=1}\lim_{l \rightarrow \infty}\frac{r_{l+1}}{r_l}=1.\end{equation}

Suppose that one of the coefficients $f_k$ is nonzero. Then $k_0$ is well-defined (see Section \ref{sectionunbounded}) and from \eqref{Recursion-konacno} it follows that 
\begin{align}
    \int_{\frac{r_{l+1}}{r_l}}^{1} \frac{1}{\tilde s^{2n+1}} \left(1+{O}(r_l\tilde s) \right) d\tilde s=f_{k_0} r_l^{n+1-k_0}\left(\frac{1+k_0-n}{(n+k_0-1)(n+1)}+o(1)\right),
    \label{Recursion-konacno-1}
\end{align}
where $o(1) \rightarrow 0$ as $r_l \rightarrow 0$.

 Notice that $$\kappa\le \frac{1}{\tilde s^{2n+1}} \left(1+{O}(r_l\tilde s) \right) \le \frac{1}{\kappa}\left(\frac{r_l}{r_{l+1}}\right)^{2n+1}, \ \ \forall\tilde s\in [\frac{r_{l+1}}{r_l},1],$$ for $\kappa>0$ small enough. This, \eqref{quotient=1} and \eqref{Recursion-konacno-1} imply  
\begin{equation} \label{finally_C}
    r_l-r_{l+1}\simeq r_l^{n+2-k_0}, \ l\to \infty,\nonumber 
\end{equation}
if $f_{k_0}(1+k_0-n)>0$. The notation $\simeq$ was introduced in Section \ref{Minkowski-def-separated}.
As $n+2-k_0>1$ (recall that $ k_0 \leq n$), we have that $(r_l)_{l\in\mathbb{N}}$ is Minkowski nondegenerate,
\begin{equation}\label{box-dim-inftyyy}
    \dim_B(r_l)_{l \in \N}=1-\frac{1}{n+2-k_0}=\frac{n+1-k_0}{n+2-k_0},
\end{equation}
and this is independent of the choice of $r_0 \in ]0,\Tilde{r}[$ (see Section \ref{proofThm1}).

When $(r_l)_{l \in \N}$, tending (monotonically) to $0$ as $l \rightarrow \infty$, is defined by
    $J_{+}(r_{l+1})-J_{-}(r_{l})=\Tilde{J}$, we assume $f_{k_0}(1+k_0-n)<0$.

\subsubsection{Completing the proof of Theorem \ref{thm4}} \label{complete-proof}
Using $x=\frac{\bar x}{r}, \ y=\frac{1}{r^{n+1}}$ we have the following relation between $F_\pm$, defined in Section \ref{sectionunbounded}, and $\Phi_{\pm}$ satisfying \eqref{phidef-} and \eqref{phidef+}:
\begin{equation}\label{y=F(x)}
\frac{1}{r^{n+1}}=F_\pm\left(\frac{\Phi_{\pm}(r)}{r}\right).
\end{equation}

The invariance of the slow divergence integral under changes of coordinates
and time reparameterizations implies that
\begin{equation}\label{invariance-int}
    J_\pm(r)=-\int_{\frac{\Phi_{\pm}(r)}{r}}^{\frac{\Phi_{\pm}(\tilde r)}{\tilde r}}\frac{F'_\pm(x)^2}{G_\pm(x)}dx, \ \forall r \in ]0,\Tilde{r}[,
\end{equation}
with $J_\pm$ defined in \eqref{J_-} and \eqref{J_+} and $G_\pm(x)=- a_{1}^{\pm}x$ (we use the change of variable $x=\frac{\Phi_{\pm}(s)}{s}$).

Assume that $I=I_+-I_-$ defined in \eqref{SDI-total} is negative on $]y^*,\infty[$, where $y^*>0$ is large enough, and let $y_0\in ]y^*,\infty[$. (We take $\tilde r$ such that $y^*=\frac{1}{\tilde r^{n+1}}$.) Then the orbit $U=\{y_0,y_1,y_2,\dots\}$ generated by $I_-(y_{l+1})=I_+(y_l)$, $l\ge 0$, tends monotonically to $+\infty$ (see Section \ref{sectionunbounded}). If we write $r_l:=\frac{1}{y_l^\frac{1}{n+1}}$ (see \eqref{Minkowski-def-inf}), then it is clear that $(r_l)_{l \in \N}$ tends monotonically to $0$, and
\begin{align*}
   I_+(y_l)-I_-(y_{l+1})&= I_+\left(\frac{1}{r_l^{n+1}}\right)-I_-\left(\frac{1}{r_{l+1}^{n+1}}\right)\\
   &=-\int_{\frac{\Phi_{+}(r_l)}{r_l}}^{0}\frac{F'_+(x)^2}{G_+(x)}dx+\int_{\frac{\Phi_{-}(r_{l+1})}{r_{l+1}}}^{0}\frac{F'_-(x)^2}{G_-(x)}dx\\
   &=J_+(r_l)-\int_{\frac{\Phi_{+}(\tilde r)}{\tilde r}}^{0}\frac{F'_+(x)^2}{G_+(x)}dx-J_-(r_{l+1})+\int_{\frac{\Phi_{-}(\tilde r)}{\tilde r}}^{0}\frac{F'_-(x)^2}{G_-(x)}dx\\
   &=J_+(r_l)-J_-(r_{l+1})+I\left(\frac{1}{\tilde r^{n+1}}\right)
\end{align*}
where we use \eqref{y=F(x)} and \eqref{invariance-int}. This implies that $(r_l)_{l \in \N}$ is generated by $J_{+}(r_l)-J_{-}(r_{l+1})=\Tilde{J}$, where $\Tilde J:=-I\left(\frac{1}{\tilde r^{n+1}}\right)$ and we can therefore use the results of Section \ref{subsubsec-analysis}. Statement 1 (resp. Statement 2) of Theorem \ref{thm4} follows from \eqref{Minkowski-def-inf} and the case $a_{1}^{-}\ne a_{1}^{+}$ (resp. $a_{1}^{-}=a_{1}^{+}$) in Section \ref{subsubsec-analysis}. (If $I$ is positive on $]y^*,\infty[$, then we use $I_-(y_{l})=I_+(y_{l+1})$ and $J_{+}(r_{l+1})-J_{-}(r_{l})=\Tilde{J}$.) This completes the proof of Theorem \ref{thm4}.

\section{Crossing limit cycles and Minkowski dimension}\label{section-applications}

In this section we consider the piecewise smooth system of Li\'{e}nard equations
\begin{equation}\label{eq-pws-lienard-general}
\left\{
\begin{array}{rcl}
  \dot{x} & = & y - F_{-}(x), \\
  \dot{y} & = & \epsilon^{2}(\epsilon\alpha_{-}+G_-(x)),
  \end{array}
\right. \  x\leq 0, \quad 
\left\{
\begin{array}{rcl}
  \dot{x} & = & y - F_{+}(x), \\
  \dot{y} & = & \epsilon^{2}(\epsilon\alpha_{+}+G_+(x)),
  \end{array}
\right. \ x\geq 0,
\end{equation}
where $0<\epsilon \ll 1$, $\alpha_\pm$ are regular parameters kept near $0$ and $C^\infty$-smooth functions $F_\pm$ and $G_\pm$ satisfy \eqref{assum1} (see Section \ref{sec-Motivation}). We define 
\begin{equation}\label{const-beta}\beta_\pm:=-\frac{2G_\pm'(0)}{F_\pm''(0)}>0.
\end{equation}

A natural question that arises is how to link the Minkowski dimension of orbits defined in Sections \ref{sectionHopf} to \ref{sectionunbounded} with the number of crossing limit cycles that \eqref{eq-pws-lienard-general} can have for $\epsilon > 0$. If the Minkowski dimension of orbits tending monotonically to $y=0$ is trivial or, equivalently, $\beta_-\ne\beta_+$ (see Theorem \ref{thm2} and note that $\beta_-\ne\beta_+$ if, and only if, $m_0(\bar G)=1$), then there are no limit cycles near the PWS Hopf point (for the precise statement see Proposition \ref{prop-no-LC} in Section \ref{sec-blow-up}). The condition $\beta_-\ne\beta_+$ means that, after blowing-up the vector field $\eqref{eq-pws-lienard-general} + 0\frac{\partial}{\partial \epsilon}$, the connection on the blow-up locus between the attracting branch $\{y = F_{+}(x)\}$ and the repelling branch $\{y = F_{-}(x)\}$ does not exist (see Figure \ref{fig-pws-lienard}).

In Section \ref{section-bounded-LC} we show that trivial Minkowski dimension of orbits tending monotonically to a balanced bounded canard cycle (see Section \ref{sectionbounded}) is not equivalent to $\beta_-\ne\beta_+$. Indeed, we find a system \eqref{eq-pws-lienard-general} undergoing a saddle-node bifurcation of limit cycles and a system without limit cycles in which, in both cases, the Minkowski dimension is trivial. Finally, in Section \ref{sec-LC-unbound} we provide examples of PWS classical Li\'{e}nard equations \eqref{eq-pws-lienard-general} such that the Minkowski dimension of the unbounded canard cycle $\Gamma_{\infty}$ is trivial.

It is convenient to set $\epsilon^2$ in \eqref{eq-pws-lienard-general} instead of $\epsilon$ when we introduce a family blow-up at the PWS Hopf point (see Section \ref{sec-blow-up}). 

\subsection{Limit cycles near the PWS Hopf point}\label{sec-blow-up}


Our goal is to study limit cycles of \eqref{eq-pws-lienard-general} in a small $\epsilon$-uniform neighborhood of the origin $(x,y)=(0,0)$ (i.e., the neighborhood does not shrink to the origin as $\epsilon\to 0$).
For this purpose, we start our analysis by performing a blow-up transformation $$(x,y,\epsilon) = (\rho\bar{x},\rho^{2}\bar{y},\rho\bar{\epsilon}),$$
with $(\bar{x},\bar{y},\bar{\epsilon})\in\mathbb S^2$, $\rho\ge 0$ and $\bar\epsilon\ge 0$. We work with different charts. We point out that only the phase directional chart $\{\bar y=1\}$ and the family chart $\{\bar \epsilon=1\}$ are relevant  for the study of limit cycles of \eqref{eq-pws-lienard-general}, produced by the PWS Hopf point or (bounded and unbounded) canard cycles (see Figure \ref{fig-pws-lienard}).

\subsubsection{Dynamics in the chart $\bar{y} = 1$}\label{sec-phase-direc-y}

Using the coordinate change $(x,y,\epsilon) = (\rho\bar{x},\rho^{2},\rho\bar{\epsilon})$, with small $\rho\ge 0$ and $\epsilon\ge 0$ and $\bar x$ kept in a large compact subset of $\mathbb R$, we obtain (after division by $\rho$) 
\begin{equation}\label{eq-blowup-y1}
\left\{
\begin{array}{rcl}
  \dot{\bar{x}} & = & 1-\frac{F_\pm''(0)}{2}\bar{x}^{2} +O(\rho\bar x^3)-\frac{1}{2}\bar x\bar\epsilon^2\left(\bar\epsilon\alpha_\pm+G_\pm'(0)\bar x+O(\rho\bar x^2)\right), \\
  \dot{\rho} & = &\frac{1}{2}\rho\bar\epsilon^2 \left(\bar\epsilon\alpha_\pm+G_\pm'(0)\bar x+O(\rho\bar x^2)\right), \\
  \dot{\bar{\epsilon}} & = & -\frac{1}{2}\bar\epsilon^3\left(\bar\epsilon\alpha_\pm+G_\pm'(0)\bar x+O(\rho\bar x^2)\right).
  \end{array}
\right.    
\end{equation}

System \eqref{eq-blowup-y1} with $-$ (resp. $+$) is defined on $\bar x\le 0$ (resp. $\bar x\ge 0$). 
 Let us briefly describe the dynamics of \eqref{eq-blowup-y1}. In the invariant line $\{\rho = \bar{\epsilon} = 0\}$ there are two singularities $\bar{x} = \pm \sqrt{\frac{2}{F_\pm''(0)}}$, which we will denote by $p_{\pm}$. Both singularities have two-dimensional center manifolds, and the unstable (resp. stable) manifold of $p_{-}$ (resp. $p_{+}$) is the $\bar{x}$-axis. Such singularities correspond to the intersection of the critical manifold with the blow-up locus. The center behavior near $p_\pm$ is given by 
 \begin{equation}
\left\{
\begin{array}{rcl}
  \dot{\rho} & = &\frac{1}{2}\rho\bar\epsilon^2 \left(\pm G_\pm'(0)\sqrt{\frac{2}{F_\pm''(0)}}+O(\rho,\bar\epsilon)\right),\nonumber \\
  \dot{\bar{\epsilon}} & = & -\frac{1}{2}\bar\epsilon^3\left(\pm G_\pm'(0)\sqrt{\frac{2}{F_\pm''(0)}}+O(\rho,\bar\epsilon)\right).\nonumber
  \end{array}
\right.    
\end{equation}

\subsubsection{Dynamics in the chart $\bar{\epsilon} = 1$}

Using the rescaling $(x,y) = ({\epsilon}\bar{x},{\epsilon}^{2}\bar{y})$, with $(\bar x,\bar y)$ kept in a large compact set, one obtains (after division by $\epsilon$) the PWS system
\begin{equation}\label{eq-pws-blownup}
\left\{
\begin{array}{rcl}
  \dot{\bar{x}} & = & \bar{y} - \frac{F_-''(0)}{2}\bar{x}^{2} + O(\epsilon\bar x^3), \\
  \dot{\bar{y}} & = & \alpha_{-} +G_-'(0)\bar x+O(\epsilon\bar x^2),
  \end{array}
\right. \quad 
\left\{
\begin{array}{rcl}
  \dot{\bar{x}} & = & \bar{y} - \frac{F_+''(0)}{2}\bar{x}^{2} + O(\epsilon\bar x^3), \\
  \dot{\bar{y}} & = & \alpha_{+} +G_+'(0)\bar x+O(\epsilon\bar x^2).
  \end{array}
\right.
\end{equation}

Let us describe the dynamics in the top of the blow-up locus. For ${\epsilon} = 0$ and $\alpha_\pm=0$, the PWS system \eqref{eq-pws-blownup} has the following form:
 \begin{equation}\label{eq-hamiltonian}
\left\{
\begin{array}{rcl}
  \dot{\bar{x}} & = & \bar{y} - \frac{F_-''(0)}{2}\bar{x}^{2}, \\
  \dot{\bar{y}} & = & G_-'(0)\bar x,
  \end{array}
\right. \ \bar x\le 0, \quad 
\left\{
\begin{array}{rcl}
  \dot{\bar{x}} & = & \bar{y} - \frac{F_+''(0)}{2}\bar{x}^{2}, \\
  \dot{\bar{y}} & = &G_+'(0)\bar x,
  \end{array}
\right. \ \bar x\ge 0.
\end{equation}

A first integral of \eqref{eq-hamiltonian} is given by $$H_{\pm}(\bar x,\bar y) = e^{-\frac{2\bar y}{\beta_{\pm}}}(\frac{\bar{y}}{\beta_{\pm}} + G_\pm'(0) \frac{\bar{x}^{2}}{\beta_{\pm}^2} + \frac{1}{2}),$$
where $\beta_\pm$ were defined in \eqref{const-beta}. Note that $p_-$ (resp. $p_+$) defined in Section \ref{sec-phase-direc-y} is the end point of the invariant curve $\{\bar{y}= - G_-'(0) \frac{\bar{x}^{2}}{\beta_{-}} - \frac{\beta_-}{2},\bar x\le 0\}$ (resp. $\{\bar{y}= - G_+'(0) \frac{\bar{x}^{2}}{\beta_{+} }- \frac{\beta_+}{2},\bar x\ge 0 \}$), corresponding to the level $H_{-}(\bar x,\bar y)=0$ (resp. $H_{+}(\bar x,\bar y)=0$) Moreover, the curve intersects the switching locus at $(0,-\frac{\beta_{-}}{2})$ (resp. $(0,-\frac{\beta_{+}}{2})$). The origin $(\bar x,\bar y)=(0,0)$ is a center for both vector fields in \eqref{eq-hamiltonian} and it corresponds to the levels $H_{\pm}(\bar x,\bar y) = \frac{1}{2}$. The origin has a ``focus-like'' behavior for $\beta_{-} \neq \beta_{+}$ and a ``center-like'' behavior for $\beta_{-}=\beta_{+}$. See Figure \ref{fig-pws-lienard}.
\begin{figure}[htb]
	\begin{center}
    \begin{overpic}[width=0.8\textwidth]{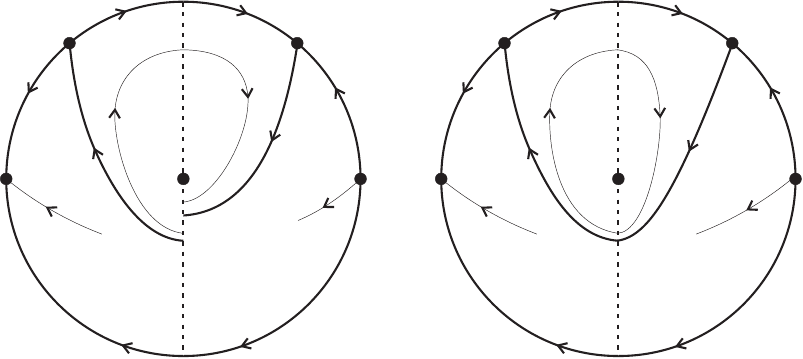}
\put(21,-4){(a)}
\put(75,-4){(b)}
\put(5,41){$p_-$}
\put(38,41){$p_+$}
\put(59,41){$p_-$}
\put(92,41){$p_+$}
\end{overpic}
         \end{center}
 \caption{Phase portraits of the blown-up vector field \eqref{eq-hamiltonian} for $\beta_{-} \neq \beta_{+}$ (Figure (a)) and $\beta_{-} = \beta_{+}$ (Figure (b)). Each of the charts $\{\bar x=\pm 1\}$ contains one extra singularity of hyperbolic type. In the chart $\{\bar y=-1\}$ the dynamics is regular pointing from the right to the left.}\label{fig-pws-lienard}
\end{figure}
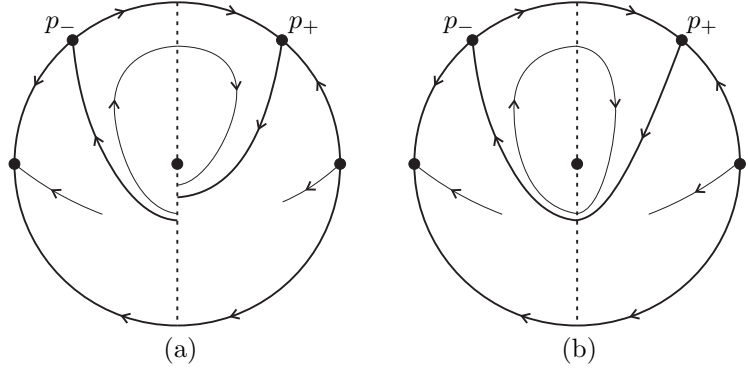

It is now clear that, for $\beta_{-}=\beta_{+}$, the study of crossing limit cycles of \eqref{eq-pws-lienard-general} in a small $\epsilon$-uniform neighborhood of the origin $(x,y)=(0,0)$ has three components: (1) the study near the origin $(\bar x,\bar y)=(0,0)$ inside the family \eqref{eq-pws-blownup}, (2) the study near the singular polycycle (it consists of $p_-$ and $p_+$ and the regular orbits that are heteroclinic to them), combining \eqref{eq-blowup-y1} and \eqref{eq-pws-blownup}, and (3) the study near ``ovals'' surrounding the origin inside the family \eqref{eq-pws-blownup}. This is a topic of further study.

Suppose that $\beta_{-}\ne\beta_{+}$. Then the connection between $p_+$ and $p_-$ is broken (see Figure \ref{fig-pws-lienard}(a)). We consider the half return maps $\Pi_{\pm}:]0,\infty[\rightarrow ]-\frac{\beta_{\pm}}{2},0[$. We define $\Pi_{+}$ by considering the flow of \eqref{eq-hamiltonian} in forward time and $\Pi_{-}$ by considering the flow of \eqref{eq-hamiltonian} in backward time.

\begin{proposition}\label{prop-pws-lienard}
Consider system \eqref{eq-hamiltonian}. If $\beta_{-} \neq \beta_{+}$, then the difference map $\Delta(\bar y) = \Pi_{+}(\bar y) - \Pi_{-}(\bar y)$ does not have zeros in the interval $]0,\infty[$.
\end{proposition}
\begin{proof}
For any $\bar y\in ]0,\infty[$, the points $(0,\bar y)$ and $(0,\Pi_{\pm}(\bar y))$ belong to the same level curve $H_{\pm}(\bar x,\bar y) = h_{\pm}$, with $0 < h_{\pm} < \frac{1}{2}$. If we write $H_\pm(\bar y):=H_{\pm}(0,\bar y)$, then we have $H_{\pm}(\bar y) = H_{\pm}( \Pi_{\pm}(\bar y) )$. This implies
\begin{equation}\label{eq-derivative-half-return}
H_{\pm}'(\bar y) = H_{\pm}'( \Pi_{\pm}(\bar y) )\Pi_{\pm}'(\bar y).
\end{equation}

From \eqref{eq-derivative-half-return} it follows that $u=\Pi_{\pm}(\bar y)$ is the $\bar y>0$-subset of the stable manifold of the hyperbolic
saddle $(u, \bar y) = (0, 0)$ of
\begin{equation}\label{eq-auxiliary-system}
\left\{
\begin{array}{rcl}
    \dot{u} & = & \bar ye^{-\frac{2\bar y}{\beta_{\pm}}}, \\
    \dot{\bar y} & = & ue^{-\frac{2u}{\beta_{\pm}}}.
\end{array}
\right.
\end{equation}

We remark that $u=\Pi_{\pm}(\bar y)$ are contained in the second quadrant $\{u<0,\bar y>0\}$. The result follows by showing that $u=\Pi_{-}(\bar y)$ and $u=\Pi_{+}(\bar y)$ do not have intersection points in the second quadrant. Suppose by contradiction that they intersect in the second quadrant. This would imply the existence of contact points between the orbits of system \eqref{eq-auxiliary-system} with $\beta_-$ and the separatrix $u=\Pi_{+}(\bar y)$. However, observe that
$$(\bar ye^{-\frac{2\bar y}{\beta_{-}}}, ue^{-\frac{2u}{\beta_{-}}}) \cdot (ue^{-\frac{2u}{\beta_{+}}},-\bar ye^{-\frac{2\bar y}{\beta_{+}}}) = u\bar y\Big{(}e^{-\frac{2\bar y}{\beta_{-}}-\frac{2u}{\beta_{+}}} - e^{-\frac{2\bar y}{\beta_{+}}-\frac{2u}{\beta_{-}}}\Big{)}.$$

For $\beta_{-} \neq \beta_{+}$, the last equation only vanishes in the set $\{u\bar y(u-\bar y) = 0\}$, and therefore there are no contact points in the second quadrant.
\end{proof}

A direct consequence of Proposition \ref{prop-pws-lienard} is the following.

\begin{proposition}\label{prop-no-LC}
Consider the PWS Li\'{e}nard equation \eqref{eq-pws-lienard-general} with $\beta_{-}\neq\beta_{+}$. Let $[\bar y_1,\bar y_2]\subset ]0,\infty[$. Then there exist  $\epsilon_{0} > 0$ and a small neighborhood $\mathcal U_\pm$ of $\alpha_\pm=0$ such that for each $(\epsilon,\alpha_-,\alpha_+)\in [0,\epsilon_0]\times \mathcal U_-\times \mathcal U_+$ the difference map $\Delta_{\epsilon,\alpha_-,\alpha_+}(\bar y)$ associated to \eqref{eq-pws-blownup} does not have zeros in $[\bar y_1,\bar y_2]$.
\end{proposition}

Observe that we do not study limit cycles bifurcating close to the origin of the phase space of \eqref{eq-pws-blownup}.

\subsection{Limit cycles near bounded canard cycles}\label{section-bounded-LC}
Consider the PWS Li\'{e}nard equation
\begin{align}\label{PWSLienard-example-balanced}
  \begin{cases}
  \dot x=y-\frac{x^{2}}{2} ,\\ 
  \dot y=\epsilon \left(-(1+\delta) x - \frac{x^{2}}{2} + x^{4}\right), \end{cases}
  \   x\le 0,\quad
  \begin{cases}
  \dot x=y-\frac{x^{2}}{2} ,\\ 
  \dot y=\epsilon \left(-x - \frac{x^{2}}{2} + x^{4}\right), \end{cases}  \  x\ge 0,
  \end{align}
with $\delta\in\mathbb R$ kept close to $0$. This section is devoted to prove the following proposition.
\begin{proposition}\label{prop-bounded-CL}
There is a continuous function $\hat y:]-\delta_0,\delta_0[\to \mathbb R$ with $\delta_0 > 0$ small and  satisfying $\hat y(0)>0$ such that, for each $\delta\in ]-\delta_0,\delta_0[$, $y=\hat y(\delta)$ is a simple zero of the slow divergence integral $I_\delta$ of \eqref{PWSLienard-example-balanced}, defined in \eqref{SDI-example-1}.
\end{proposition}

Concerning System \eqref{PWSLienard-example-balanced}, Proposition \ref{prop-bounded-CL} and Theorem \ref{thm3} (see also Remark \ref{remark-balanceddd}) imply that, for each $\delta\in ]-\delta_0,\delta_0[$, the Minkowski dimension of any entry-exit orbit tending monotonically to $\hat y(\delta)$ is equal to $0$.

It can be checked that system \eqref{PWSLienard-example-balanced} satisfies \eqref{assum1} for $\delta$ sufficiently small, and the curve of singularities is given by $S=\{y=\frac{x^{2}}{2}\}$. The associated slow dynamics is given by (see also \eqref{PWSslowdyn})
$$x'=-1-\delta-\frac{x}{2}+x^3, \text{ for } x\le 0, \ x'=-1-\frac{x}{2}+x^3, \text{ for } x\ge 0. $$ 

The slow dynamics has a simple zero $x_0>0$ and it is strictly negative for all $x\in]-x_0,x_0[$ and $\delta$ small enough. One can compute $x_{0}$ numerically, and we obtain $x_{0} \approx 1,16537$.

Using \eqref{SDI-total} we get
\begin{equation}
    \label{SDI-example-1}
    I_\delta(y)=\int_0^{\sqrt{2y}}\frac{xdx}{-1-\frac{x}{2}+x^3}+\int_{-\sqrt{2y}}^0\frac{xdx}{-1-\delta-\frac{x}{2}+x^3}, \   y\in]0,\frac{x_0^2}{2}[. 
\end{equation} 

Following \cite[Section 5.3]{BoxDomagoj}, we know that $I_0$ has a simple zero $\hat y_0\in]0,\frac{x_0^2}{2}[$, and it follows from the Implicit Function Theorem that this zero persists for $\delta > 0$ sufficiently small. One can compute zeroes of $I_{0}$ numerically. Indeed, by approximating the integrand of \eqref{SDI-example-1} using Taylor series and then evaluating the integral, one obtains $\hat y_0 \approx 0,608853$. Observe that $\frac{x_0^2}{2} \approx 0,679047$.

Now, we define
\begin{equation}\label{eq-pws-bounded}
\begin{cases}
  \dot x=y-\frac{x^{2}}{2} ,\\ 
  \dot y=\epsilon^2 \left(\epsilon\alpha_--(1+\delta) x - \frac{x^{2}}{2} + x^{4}\right), \end{cases}
  \  
  \begin{cases}
  \dot x=y-\frac{x^{2}}{2} ,\\ 
  \dot y=\epsilon^2 \left(\epsilon\alpha_+ -x - \frac{x^{2}}{2} + x^{4}\right), \end{cases}
\end{equation}
with $\alpha_{\pm}$ kept near $0$. System \eqref{eq-pws-bounded} with $-$ (resp. $+$) corresponds to the vector field defined in $x \leq 0$ (resp. $x \geq 0$). We have 
$$\beta_-=2(1+\delta) \text{ and } \beta_+=2,$$
with $\beta_\pm$ defined in \eqref{const-beta}.

Let $\hat\delta\in ]-\delta_0,\delta_0[$, $\hat\delta\ne 0$. Then \eqref{eq-pws-bounded} has no crossing limit cycles Hausdorff close to the balanced canard cycle $\Gamma_{\hat y(\hat\delta)}$, for $\epsilon>0$, $\epsilon\sim 0$, $\alpha_\pm\sim 0$ and $\delta\sim\hat\delta$. Indeed, notice that the connection on the blow-up locus between the attracting branch $S_+=\{y=\frac{x^{2}}{2},x>0\}$ and the repelling branch $S_-=\{y=\frac{x^{2}}{2},x<0\}$ of $S$ is broken, because $\beta_-\ne\beta_+$ for $\delta=\hat\delta$ (see Figure \ref{fig-pws-lienard}(a)).

Suppose that $\delta = 0$ and $\alpha_{\pm} = \alpha$. Then \eqref{eq-pws-bounded} becomes a smooth slow-fast system and therefore we are in the framework of \cite[Section 5.3]{BoxDomagoj}. Thus, for each $\epsilon>0$ and $\epsilon\sim 0$, \eqref{eq-pws-bounded} undergoes a saddle-node bifurcation of crossing limit cycles, Hausdorff close to $\Gamma_{\hat y(0)}$, when we vary $\alpha\sim 0$. 

\begin{remark}\label{remark-nonzeroMD}
If $y=\hat y$ is a zero of $I$ of multiplicity $m_{\hat y}(I)$, then \eqref{eq-pws-lienard-general} can have at most $m_{\hat y}(I)+1$ limit cycles Hausdorff close to the canard cycle $\Gamma_{\hat y}$, for $\epsilon>0$, $\epsilon\sim 0$ and $\alpha_\pm\sim 0$. Moreover, if $\beta_\pm(\delta)$ are functions of $\delta$, $\beta_-(0)=\beta_+(0)$ and $\beta_-'(0)\ne \beta_+'(0)$ (connection between $p_+$ and $p_-$ is broken in a regular way), and $I$ has a simple zero at $y=\hat y$, then for each $\epsilon>0$, $\epsilon\sim 0$ and $\alpha_\pm\sim 0$, \eqref{eq-pws-lienard-general} undergoes a saddle-node bifurcation of (crossing) limit cycles, near $\Gamma_{\hat y}$, when we vary $\delta\sim 0$. (We can apply this to \eqref{eq-pws-bounded}.) These and other cyclicity results will be proved in a separate paper.
\end{remark}

\subsection{Limit cycles near the unbounded canard cycle}\label{sec-LC-unbound}
Consider the classical PWS Li\'{e}nard equation
\begin{align}\label{PWSLienard-example-unbounded}
  \begin{cases}
  \dot x=y-(x^4+2x^2),\\ 
  \dot y=-\epsilon 2x, \end{cases}
  \   x\le 0,\quad
  \begin{cases}
  \dot x=y-(x^4+\delta x^2),\\ 
  \dot y=-\epsilon x, \end{cases}  \  x\ge 0,
  \end{align}
with $\delta\in\mathbb R$ kept close to $1$. System \eqref{PWSLienard-example-unbounded} is a special case of \eqref{model-Lienard1} and it satisfies \eqref{assum1} and \eqref{assum2} with $L_-=]-\infty,0[$ and $L_+=]0,\infty[$. Statement 1 of Theorem \ref{thm4} implies that, for each $\delta$ close to $1$, the Minkowski dimension of any entry-exit orbit tending (monotonically)
to $\infty$ is equal to $0$. 

We focus now on 
\begin{align}\label{PWSLienard-example-unbounded-perturbed}
  \begin{cases}
  \dot x=y-(x^4+2x^2),\\ 
  \dot y=\epsilon^2(\epsilon\alpha_- -2x), \end{cases}
  \   x\le 0,\quad
  \begin{cases}
  \dot x=y-(x^4+\delta x^2),\\ 
  \dot y=\epsilon^2(\epsilon\alpha_+- x), \end{cases}  \  x\ge 0,
  \end{align}
  where $\alpha_\pm$ are close to zero. We have (see \eqref{const-beta})
  $$\beta_-=1 \text{ and } \beta_+=\frac{1}{\delta}.$$

Take $\hat \delta\ne 1$. Then $\beta_-\ne \beta_+$ and \eqref{PWSLienard-example-unbounded-perturbed} has no limit cycles Hausdorff close to the unbounded canard cycle defined in Section \ref{m<-section-infty}, for $\epsilon>0$, $\epsilon\sim 0$, $\alpha_\pm\sim 0$ and $\delta\sim\hat\delta$ (see Figure \ref{fig-pws-lienard}(a)). For $\delta=1$, we have the orbit on the blow-up locus connecting $p_+$ and $p_-$ (see Figure \ref{fig-pws-lienard}(b)), and the unbounded canard cycle may produce limit cycles of \eqref{PWSLienard-example-unbounded-perturbed} for $\epsilon>0$, $\epsilon\sim 0$, $\alpha_\pm\sim 0$ and $\delta\sim 1$.

\section*{Declarations}
 
\textbf{Ethical Approval} \ 
Not applicable.
 \\
\\
 \textbf{Competing interests} \  
The authors declare that they have no conflict of interest.\\
 \\
\textbf{Authors' contributions} \  All authors conceived of the presented idea, developed the theory, performed the computations and
contributed to the final manuscript.  \\ 
\\ 
\textbf{Funding} \
The research of R. Huzak and G. Radunovi\'{c}  was supported by: Croatian Science Foundation (HRZZ) grant IP-2022-10-9820. Additionally, the research of G. Radunovi\'{c} was partially
supported by the Horizon grant 101183111-DSYREKI-HORIZON-MSCA-2023-SE-01. Otavio Henrique Perez is supported by Sao Paulo Research Foundation (FAPESP) grants 2021/10198-9 and 2024/00392-0.\\
 \\
\textbf{Availability of data and materials}  \
Not applicable.

\bibliographystyle{plain}
\bibliography{bibtex}
\end{document}